\documentclass[a4paper,12pt]{amsart}
\usepackage{amsmath}
\usepackage{amsfonts}	% \mathscr \mathbb \mathfrak
\usepackage{amssymb}  	% special symbols
\usepackage{mathrsfs}
\usepackage{mathtools}	% e.g., \coloneqq
\usepackage{savesym}
\usepackage{esvect}		% \vv
\usepackage{ifthen}		% \ifthenelse
\usepackage{color} 	 	% colors
\usepackage{graphicx}  	% include figures
\usepackage{caption}  	% caption dealing
\usepackage{amscd} 	  	% basic commutative diagrams; complex ones: kuvio, XY-pic
\usepackage{float} 	  	% table figure positioning
\usepackage[shortlabels,inline]{enumitem} % smart enumeration
\usepackage{tikz}    % draw general diagrams
\usepackage{tikz-cd} % lib for powerful commutative diagrams

\usepackage[centering,width=6in]{geometry}
\usepackage[numbers,sort]{natbib} 	% automatically sort the citations

\usepackage{aliascnt} % for the correct \autoref
\usepackage[pdfusetitle,
		bookmarks=true,
		pagebackref=false,
		colorlinks,
		linkcolor=blue,
		citecolor=blue,
		urlcolor=black]{hyperref}

\usepackage{amsthm}

\theoremstyle{plain}
\newaliascnt{theorem}{dummy}
\newtheorem{theorem}[theorem]{Theorem}
\aliascntresetthe{theorem}

\newaliascnt{proposition}{dummy}
\newtheorem{proposition}[proposition]{Proposition}
\aliascntresetthe{proposition}

\newaliascnt{corollary}{dummy}
\newtheorem{corollary}[corollary]{Corollary}
\aliascntresetthe{corollary}

\newaliascnt{lemma}{dummy}
\newtheorem{lemma}[lemma]{Lemma}
\aliascntresetthe{lemma}

\newaliascnt{conjecture}{dummy}
\newtheorem{conjecture}[conjecture]{Conjecture}
\aliascntresetthe{conjecture}

\theoremstyle{definition}
\newaliascnt{definition}{dummy}
\newtheorem{definition}[definition]{Definition}
\aliascntresetthe{definition}

\newaliascnt{example}{dummy}

\aliascntresetthe{example}

\newaliascnt{remark}{dummy}
\newtheorem{remark}[remark]{Remark}
\aliascntresetthe{remark}

\theoremstyle{remark}

\numberwithin{equation}{section} % \numberwithin also works for other counters

\allowdisplaybreaks

\newcommand{\calA}{\mathcal{A}}

\newcommand{\calD}{\mathcal{D}}
\newcommand{\calF}{\mathcal{F}}

\newcommand{\calP}{\mathcal{P}}

\newcommand{\calS}{\mathcal{S}}
\newcommand{\bbQ}{\mathbb{Q}}
\newcommand{\bbR}{\mathbb{R}}
\newcommand{\bbC}{\mathbb{C}}
\newcommand{\bbZ}{\mathbb{Z}}
\newcommand{\bbN}{\mathbb{N}}
\newcommand{\bbP}{\mathbb{P}}

\newcommand{\bbK}{\mathbb{K}}
\DeclareMathOperator*{\intxn}{\cap}

\DeclareMathOperator*{\union}{\cup}
\DeclareMathOperator*{\Union}{\bigcup}

\providecommand{\abs}[1]{\lvert#1\rvert}
\providecommand{\Abs}[1]{\left\lvert#1\right\rvert}
\providecommand{\fl}[1]{\lfloor#1\rfloor}

\providecommand{\fl}[1]{\lfloor#1\rfloor}

\providecommand{\Ceil}[1]{\left\lceil #1\right\rceil}
\providecommand{\norm}[2][]{\lVert#2\rVert_{#1}}

\newcommand{\euclid}[1][d]{\mathbb{R}^{#1}}

\newcommand{\wt}[1]{\widetilde{#1}}
\newcommand{\wh}[1]{\widehat{#1}}

\newcommand{\mFor}{\quad\text{for }}
\newcommand{\mAnd}{\quad\text{ and }\quad}
\newcommand{\mForall}{\quad\text{for all }}
\newcommand{\tUB}{\min\left\{1,H(p)/\chi\right\}}
\newcommand{\Alg}{\overline{\mathbb{Q}}}
\newcommand{\AlgInt}{\overline{\mathbb{Z}}}

\DeclareMathOperator{\diag}{diag}

\newcommand{\GLR}[1][d]{\mathrm{GL}_{#1}(\mathbb{R})}

\newcommand{\Ogrp}[1]{\mathrm{O}_{#1}(\mathbb{R})}
\newcommand{\Ugrp}[1]{\mathrm{U}_{#1}(\mathbb{C})}
\newcommand{\MatR}[1][2d]{\mathrm{M}_{#1}(\mathbb{R})}
\newcommand{\MatC}[1][d]{\mathrm{M}_{#1}(\mathbb{C})}

\title{Dimension of homogeneous iterated function systems with algebraic translations}

\author{De-Jun Feng}
\address{
	Department of Mathematics\\
	The Chinese University of Hong Kong\\
	Shatin,  Hong Kong
}
\email{djfeng@math.cuhk.edu.hk}

\author{Zhou Feng}
\address{
	Department of Mathematics\\
	The Chinese University of Hong Kong\\
	Shatin,  Hong Kong
}
\email{zfeng@math.cuhk.edu.hk}

\date{}
\dedicatory{}
\thanks{}

\begin{document}
	
\begin{abstract}
 Let $ \mu $ be the self-similar measure associated with a homogeneous iterated function system $ \Phi = \{ \lambda x + t_j \}_{j=1}^m $ on $\bbR$ and  a probability vector $ (p_{j})_{j=1}^m$, where $0\neq \lambda\in (-1,1)$  and $t_j\in \bbR$.   Recently by modifying the arguments of  Varj\'u in \cite{Varju2019},   Rapaport and Varj\'u \cite{RapaportVarju2024}  showed that if $t_1,\ldots, t_m$ are rational numbers and $0<\lambda<1$, then $$ \dim \mu =\min\left \{ 1, \; \frac{\sum_{j=1}^m p_{j}\log p_{j}}{ \log |\lambda| }\right\}$$
 unless $ \Phi $ has exact overlaps.  In this paper, we further show that the above equality holds in the case when $t_1,\ldots, t_m$ are algebraic numbers and $0<|\lambda|<1$.  This is done by adapting and extending the ideas of Breuillard,  Rapaport and Varj\'u in \cite{BreuillardVarju2019,BreuillardVarju2020,RapaportVarju2024,Varju2019}.
 % including an extension of a relation between Garsia entropy and Mahler measure to the setup of number fields.
\end{abstract}

\maketitle

\section{Introduction}\label{sec:intro}

\subsection{Background and our main result}\label{subsec:BgMT}

 In this paper, by an \textit{iterated function system} (IFS) we mean a finite family $ \Phi = \{\varphi_{j}\}_{j=1}^m $ of contracting similarities on $ \bbR $, taking the form
\begin{equation*}\label{eq:IFS}
	\varphi_{j}(x) = \lambda_{j} x + t_{j},
\end{equation*}
where $ 0\neq \lambda_{j} \in (-1, 1) $ and $ t_{j} \in \bbR $ for $1\leq j\leq  m$. It is well known~\cite{Hutchinson1981} that there exists a unique nonempty compact set $ K \subset \bbR $ such that
\begin{equation*}\label{eq:selfsim-set}
	K = \Union_{j=1}^m \varphi_{j}(K).
\end{equation*}
We call $ K $ the \textit{self-similar set} generated by $ \Phi $. Let $ p = (p_{j})_{j=1}^m $ be a probability vector  with strictly positive
entries. By \cite{Hutchinson1981} there exists a unique Borel probability measure $ \mu $ on $ \bbR $ such that
\begin{equation*}
	\mu = \sum_{j=1}^m p_{j}  \mu\circ \varphi_{j}^{-1}.
\end{equation*}
Actually $\mu$ is fully supported on $K$.
 We call $ \mu $ the \textit{self-similar measure} associated with $ \Phi $ and $ p $. It is known \cite{FengHu2009} that  $ \mu $ is  exact dimensional, in the sense that the limit
$$
\lim_{r\to 0}\frac{\log \mu([x-r,x+r])}{\log r}
$$
exists and is equal to a constant for $\mu$-almost every $x$; we write $\dim \mu$ for  this constant and call it  the {\it dimension} of $\mu$.

The dimension theory of self-similar measures is a central topic in fractal geometry. Let $\mu=\mu_{\Phi, p}$ be the self-similar measure associated with an IFS $\Phi=\{\varphi_{j}(x)=\lambda_jx+t_j\}_{j=1}^m $ on $\bbR$ and a probability vector $p$ with strictly positive entries.   There is a natural upper bound for $ \dim \mu $ in terms of the entropy $ H(p) $ and the Lyapunov exponent $ \chi $. That is, writing
\begin{equation}\label{eq:def-DynQuantities}
	H(p) := \sum_{j=1}^m - p_{j} \log p_{j} \mAnd \chi := - \sum_{j=1}^m p_{j} \log \abs{\lambda_{j}},
\end{equation}
 one has
\begin{equation}\label{eq:naturalUB}
	\dim \mu \leq \min \left \{1,\; \frac{H(p)}{\chi} \right \};
\end{equation}
 see e.g.~\cite[Corollary 5.2.3]{Edgar1998}.
It turns out that the equality in \eqref{eq:naturalUB} holds in many cases, such as when  the IFS $\Phi$ satisfies the open set condition (see e.g. \cite[ Theorem 5.2.5]{Edgar1998}). Here $\Phi$ is said to satisfy the {\it open set condition} if there is a nonempty open set $U\subset \bbR$ such that $\varphi_j(U)$,  $1\leq j\leq m$, are disjoint subsets of $U$.
%Moreover, Jordan, Pollicott, and Simon~\cite{JordanEtAl2007} showed that if the contracting ratios $ (\lambda_{j})_{j\in\Lambda} $ are fixed such that $ 0 < \abs{\lambda_{j}} < 1/2 $ for $ j \in\Lambda $, then the equality in \eqref{eq:naturalUB} holds for Lebesgue almost all translations $ (t_{j})_{j\in\Lambda} $.
On the other hand, there are  some circumstances in which the dimension drop (that is, $ \dim \mu < \tUB $) can occur. For $ n \in \bbN $ and
$ J = j_{1}\ldots j_{n} \in \{1,\ldots, m\}^{n} $, write for brevity
\begin{equation*}
	\varphi_{J} = \varphi_{j_{1}} \circ \cdots \circ \varphi_{j_{n}} \mAnd \lambda_{J} = \lambda_{j_{1}} \cdots \lambda_{j_{n}}.
\end{equation*}
The IFS $ \Phi $ is said to have \textit{exact overlaps}  if there exist $n$ and  distinct $ J_{1}, J_{2}\in \{1,\ldots, m\}^{n} $ such that $ \varphi_{J_{1}} =  \varphi_{J_{2}}$.  It is not difficult to see that $ \dim \mu < \tUB $ whenever $ \Phi $ has exact overlaps and $ \dim\mu < 1 $. The following folklore conjecture, which is also called the \textit{exact overlaps conjecture}, asserts that these are  the only  circumstances in which the dimension  drop can occur.

\begin{conjecture}\label{conj:eoc}
	Let $\mu$ be the self-similar measure associated with the IFS $\Phi$ on $\bbR$  and the probability vector $p$. If  $ \Phi $ has no exact overlaps, then
	$$
	\dim \mu = \min \left \{1,\; \frac{H(p)}{\chi} \right \}.
	$$
\end{conjecture}

 A version of this conjecture for the dimension of self-similar sets was first stated  by Simon~\cite{Simon1996}.

 In recent years, significant progress has been made towards Conjecture \ref{conj:eoc}. The first breakthrough was achieved by Hochman~\cite{Hochman2014},   who proved Conjecture \ref{conj:eoc} if the IFS $\Phi$ satisfies  the exponential separation condition.   To introduce this separation condition, for $n\in \bbN$ define
$$
\Delta_n=\min\left\{|\varphi_{w_1}(0)-\varphi_{w_2}(0)|:\; w_1, w_2\in \{1,\ldots, m\}^n, w_1\neq w_2 \mbox{ and } \lambda_{w_1}=\lambda_{w_2}\right\}.
$$
The IFS $\Phi$ is said to satisfy the {\it exponential separation condition} if there exists $c>0$ such that $\Delta_n\geq c^n$ for infinitely many $n$.  The exponential separation condition is satisfied by broad families of
ovelapping IFSs. For instance, it is satisfied by every algebraic IFS $\Phi$ on $\Bbb R$ (i.e., all the parameters $\lambda_j$ and $t_j$  are algebraic numbers) that does not allow
exact overlaps;  see \cite[Theorem 1.5]{Hochman2014}.

As a far-reaching extension of Hochman’s result, Rapaport \cite{Rapaport2022} recently
established the conjecture if only the contraction parameters
$\lambda_j$ are assumed to be algebraic. Another important breakthrough was made by Varj\'u ~\cite{Varju2019},  who showed that if $\mu$ is a Bernoulli convolution (that is, $\mu$ is the self-similar measure associated with the IFS $\Phi_\lambda=\{\lambda x, \lambda x+1\}$ and the probability vector $p=(1/2, 1/2)$, where $1/2<\lambda<1$),
then $ \dim \mu = 1 $ if $\lambda$ is transcendental. Combined with Hochman's result regarding IFSs with algebraic parameters, this verifies Conjecture \ref{conj:eoc} for the family of Bernoulli convolutions.

 Very recently,  Rapaport and Varj\'u~\cite{RapaportVarju2024}  investigated  homogeneous IFSs of three maps.  Among other things, they proved that for each $(\lambda,\tau)\in (2^{2/3},1)\times \bbR$,  Conjecture \ref{conj:eoc} holds for the IFS $\Phi_{\lambda,\tau}=\{\lambda x,\; \lambda x+1,\; \lambda x+\tau\}$ with equal probability weights.
%Here an IFS $ \Phi $ as in \eqref{eq:IFS} is called \textit{homogeneous} if there is $ \lambda \in \bbR $ such that $ \lambda_{j} = \lambda $ for all $ 1\leq j\leq m$.
 Moreover, by modifying the arguments of  Varj\'u in \cite{Varju2019},  Rapaport and Varj\'u  showed that Conjecture \ref{conj:eoc} holds for every  homogeneous IFS $\Phi=\{\lambda x+t_j\}_{j=1}^m$ with positive contraction ratio $\lambda$ and  rational translations $t_j$ (see \cite[Theorem A.8]{RapaportVarju2024}).

 We remark that a stronger version of Conjecture \ref{conj:eoc} involving the
$L^q$ dimension (for $q\geq 1$) instead of  dimension of measures was established
 by Shmerkin \cite{Shmerkin2019} under the exponential separation condition. Subsequently, the $L^q$ dimension (for $0<q<1$) of self-similar measures on $\bbR$  was also determined by Barral and the first author \cite{BarralFeng2021} under the same separation condition. It is worth pointing out that there are some IFSs without exact overlaps for
which  the exponential separation condition fails; see \cite{Baker2021,BaranyKaeenmaeki2021} and also \cite{Baker2022b,Chen2021}.

 The main result of this paper is the following theorem,  which verifies Conjecture   \ref{conj:eoc}  for homogeneous IFSs with algebraic translations. It is a natural generalization of \cite[Theorem A.8] {RapaportVarju2024}.

\begin{theorem}\label{thm:main}
 Let $ \mu $ be the self-similar measure associated with a homogeneous IFS $ \Phi = \{ \lambda x + t_{j} \}_{j=1}^m$ on $\bbR$ and a probability vector $ p $, where $0\neq \lambda\in (-1, 1)$. Assume that  all the translations $t_j$ are algebraic numbers.    If $ \Phi $ has no exact overlaps, then
\begin{equation*}
	\dim \mu= \min \left \{1, \;\frac{H(p)}{- \log \abs{\lambda}}\right \}.
\end{equation*}
\end{theorem}

\subsection{ About the proof}

Our proof of Theorem \ref{thm:main} is adapted from   Rapaport and Varj\'u ~ \cite[Appendix A]{RapaportVarju2024} and Varj\'u \cite{Varju2019}.  Below we present two of the main ingredients used in the proof, which are the variants of the corresponding results in \cite{BreuillardVarju2019,BreuillardVarju2020,RapaportVarju2024, Varju2019} in our setting.

Throughout this subsection, let $\Phi=\{\lambda x+t_j\}_{j=1}^m$ be a fixed homogeneous IFS with $0\neq \lambda\in (-1,1)$ and $t_j$ being algebraic numbers, and
let  $ p = (p_{j})_{j=1}^m$  be a fixed probability vector. Write
\begin{equation*}
\label{e-D} \calD = \{ t_{i} - t_{j}\colon 1\leq i, j \leq m \}.
\end{equation*}
For $ n \in \bbN $, let $\calP^{n}$ denote the set of  polynomials in one variable of degree not exceeding $n$ and with coefficients in $\calD$.  That is,
\begin{equation}\label{eq:def-calPn}
	\calP^{n} = \left \{ \sum_{i=0}^{n} a_{i}X^{i} \colon  a_{i}\in \calD \mbox{ for all } 0\leq i\leq n \right \}.
\end{equation}
For $ \alpha > 0 $ and $n\in \bbN$, write
\begin{equation*}
	E^{n}_{\alpha} = \left \{ \eta \in (-1,1)\setminus \{0\} :\dim\mu_{\eta} < \alpha \text{ and } P(\eta) = 0 \text{ for some } 0 \neq P \in \calP^{n} \right \},
\end{equation*}
where $ \mu_{\eta} $ denotes the self-similar measure associated with the IFS $ \{\eta x + t_{j}\}_{j=1}^m $ and the probability vector  $ p $.

  The first main ingredient of our proof is the following theorem, which is a variant of \cite[Theorem 1]{BreuillardVarju2019} and \cite[Theorem A.2]{RapaportVarju2024} for homogeneous IFSs with algebraic translations.  It states roughly that if there is a dimension drop for $\mu_\lambda$, then there are fast algebraic approximations $ (\eta_{n}) $ to $ \lambda $ for infinitely many $n$ such that $\dim \mu_{\eta_{n}}$ also drop.

\begin{theorem}\label{thm:EAA}
	 Suppose that $ \dim \mu_{\lambda} < \min \{ 1, - H(p)/\log \abs{\lambda} \} $. Then for every $ \varepsilon >0 $ and $ n_{0} \geq 1 $, there exist $ n \geq n_{0} $ and $ \eta \in  E^{n}_{\dim \mu_{\lambda} + \varepsilon} $ such that
	\begin{equation*}
		\lvert\lambda -\eta\rvert\leq \exp(-n^{1/\varepsilon}).
	\end{equation*}
\end{theorem}

To present another ingredient of our proof, let us first introduce some notation and definitions. Let $\nu$ be an atomic measure defined by $\nu = \sum_{j=1}^m p_{j}\delta_{t_{j}}$. For $\eta\in \bbC$ with $0<|\eta|<1$, the {\it Garsia entropy $h_{\eta,\nu}$} of $\eta$ and $\nu$ is defined by
\begin{equation}
\label{e-Garsia}
h_{\eta,\nu}=\lim_{n\to \infty}\frac{H(\sum_{j=0}^{n-1}\xi_j \eta^j)}{n},
\end{equation}
where $(\xi_n)_{n=0}^\infty$ is a sequence of independent random variables with common law $\nu$, and $H(\cdot)$ stands for the Shannon entropy of a discrete random variable (see \eqref{eq:def-Shannon}). The above limit always exists,  following from a subadditive argument.

Recall that if $f(X)=a_n\prod_{i=1}^n(X-\beta_i)=\sum_{i=0}^n a_iX^i\in \bbZ[X]$ is the minimal polynomial of an algebraic number $\beta$, then the {\it Mahler measure} of $\beta$ is defined as
\begin{equation}
\label{e-Mahler}
M(\beta)=|a_n|\prod_{i=1}^n\max\{1, |\beta_i|\}.
\end{equation}

The following theorem is a variant of \cite[Theorem 9]{Varju2019} for homogeneous IFSs with algebraic translations, which gives a natural relation between the Garsia entropy and the Mahler measure. We remark that \cite[Theorem 9]{Varju2019} directly follows from \cite[Proposition 13]{BreuillardVarju2020}.

%\autoref{thm:UBGarsia} is an analog of \cite[Theorem 9]{Varju2019} and shows that for $ 0 \neq \eta \in \Alg $ with modulus less than $1$, the \textit{Garsia entropy} $ h_{\eta,\nu} $ (see \autoref{def:GarsiaEntropy}) can control the \textit{Mahler measure} $ M(\eta) $ (see \autoref{def:MahlerMeasure}).  {\color{blue} to be modified}

\begin{theorem} \label{thm:UBGarsia}    For  any $h\in (0,  H(p)) $, there is a positive number $ M  $ depending on $h$,  $ t_1,\ldots, t_m$ and  $p_1,\ldots, p_m$, such that $ h_{\eta, \nu} \geq h $ for every  algebraic number $\eta\in (-1,1)$ with $ M(\eta)\geq M $.
\end{theorem}

By assuming Theorems \ref{thm:EAA} and \ref{thm:UBGarsia}, we can prove Theorem \ref{thm:main} by simply following  the arguments in \cite[Theorem A.8]{RapaportVarju2024} with slight modifications. Below we briefly introduce the ideas for the proofs of Theorems \ref{thm:EAA} and \ref{thm:UBGarsia}. The details will be given in Sections \ref{sec:pf-EAA} and \ref{sec:UBGarsia}.

To prove Theorem \ref{thm:EAA}, we first establish a separation result (see Proposition \ref{prop:SepRoots}), stating that  if $\lambda_1,\lambda_2$ are distinct algebraic numbers each of which is a root of a polynomial in $\mathcal P^n$, then $$|\lambda_1-\lambda_2|\geq 2 n^{-Mn} \quad \mbox{if }n\geq N,$$
 where $M,N$ are two positive constants depending on $t_1,\ldots, t_m$. This is a slight generalization of \cite[Lemma 4.1]{RapaportVarju2024} (which is due to Mahler \cite{Mahler1964},  giving a similar lower bound $2 n^{-5n}$ for the distance between roots of integer polynomials with bounded coefficients). Thanks to this separation result, Theorem \ref{thm:EAA} can be proved by  following the exact proof of \cite[Theorem A.2]{RapaportVarju2024} except for some minor modifications of the involved auxiliary lemmas and propositions.

To prove Theorem  \ref{thm:UBGarsia}, we may assume that $\eta$ is a root of a polynomial in $\bigcup_{n=1}^\infty {\mathcal P}^n$; otherwise $h_{\eta,\nu}=H(p)$ and there is nothing to prove. Next we choose a real algebraic integer $\theta$ such that $${\bbQ}(\theta)=\bbQ({\mathcal D}),$$ where for  $Z\subset \bbC$, $\bbQ(Z)$ stands for the smallest subfield of $\bbC$ containing both $\bbQ$ and  $Z$.  The existence of such $\theta$ follows from Lemma \ref{lem-extension}. Set $d=\deg(\theta)$ and let $\theta_i$, $i=1, \ldots, d$, be the algebraic conjugates of $\theta$ over $\bbQ$. We manage to show that there is a constant $C$ (depending only on $\mathcal D$ and $\theta$) such that
\begin{equation}
\label{e-CM}
\prod_{i=1}^d \widetilde{M}_{\bbQ(\theta_i)}(\eta_i)\geq C M(\eta),
\end{equation}
where $\eta_i$ ($i=1,\ldots, d$) are some suitably chosen algebraic numbers with
$$
[\bbQ(\theta_i,\eta_i):\bbQ(\theta_i)]<\infty,
$$
and $\widetilde{M}_{\bbQ(\theta_i)}(\eta_i)$ are defined as
$$
\widetilde{M}_{\bbQ(\theta_i)}(\eta_i)=\prod_{j} \frac{1}{\min\{1,|\beta_j|\}},
$$
where the product is taken over all algebraic conjugates $\beta_j$ of $\eta_i$ over the field $\bbQ(\theta_i)$, including $\eta_i$ itself. Due to \eqref{e-CM}, we may choose $k_0\in \{1,\ldots, d\}$ such that
\begin{equation}
\label{e-CM2}
 \widetilde{M}_{\bbQ(\theta_{k_0})}(\eta_{k_0})\geq C^{1/d} M(\eta)^{1/d}.
\end{equation}
Assume that  $M(\eta)$ is large enough (saying, greater than $C^{-1/d}$). Then the left hand side of \eqref{e-CM2} is larger than $1$, so replacing $\eta_{k_0}$ by one of its algebraic conjugates over $\bbQ(\theta_{k_0})$ if necessary, we may assume that $$|\eta_{k_0}|<1.$$
Notice that $\eta_{k_0}$ might take value in $\bbC\setminus \bbR$.

To conclude Theorem  \ref{thm:UBGarsia}, we manage to extend a deep result of Breuillard and Varj\'u (see \cite[Proposition 13] {BreuillardVarju2020}) on the relation between the  Garsia entropy and the Mahler measure from homogeneous IFSs on $\bbR$ with rational translations to  homogeneous IFSs on $\bbC$ with complex algebraic translations; see Proposition \ref{prop:LB_Phi}. From this result, we are able to derive that
$$
h_{\eta, \nu}\geq \Phi_{\nu'}\left(\widetilde{M}_{\bbQ(\theta_{k_0})}(\eta_{k_0})\right)
$$
for a certain atomic probability measure $\nu'$  on $\bbC$, where $\Phi_{\nu'}$ is defined as in \eqref{eq:def-Phi_nu} (in which $\nu$ is replaced by $\nu'$).  Since $\Phi_{\nu'}(a)$ tends to $H(p)$ as $a\to \infty$, the above inequality and \eqref{e-CM2} yield  Theorem  \ref{thm:UBGarsia}.

We actually prove a more general version of Theorem \ref{thm:UBGarsia} in which $t_1,\ldots, t_m$ and $\eta$ are assumed to be complex algebraic numbers (see Theorem \ref{thm:UBGarsiaC}), which  might be helpful for the study of  homogeneous self-similar measures on  $\bbR^2$ or higher dimensional spaces.

%Let us briefly give some ideas of proving \autoref{thm:main}. By Hochman's result regarding IFSs with algebraic parameters, we can assume that $ \lambda $ is transcendental, thus $ \Phi $ has no exact overlaps by $ \{t_{j}\}_{j\in\Lambda} \subset \Alg $. Suppose on the contrary that $ \dim \mu_{\lambda} $ drops. Then \autoref{thm:EAA} gives algebraic numbers $ (\eta_{n}) $ such that $ \dim \mu_{\eta_{n}} $ also drops and $ \eta_{n} $ approximates $ \lambda $ in a fast way. By \autoref{prop:UBMahler}, there is an upper bound on the Mahler measures $ M(\eta_{n}) $ by some $ M > 0 $ independent of $ n $, which restricts the algebraic complexities of $ (\eta_{n}) $. Then such an approximation $ (\eta_{n}) $ with bounded Mahler measures contradicts the transcendence of $ \lambda $. Specifically, based on the transcendence of $ \lambda $ and Hochman's result about super-exponentially small values of polynomials (see \autoref{thm:SuperExpClose}), it follows from \autoref{prop:algLargeOrder} that $ \eta_{n} $ with $n$ large should be a zero of some polynomial $P \in \Union_{k=1}^{\infty}\calP^{k} $. However, the order of $\eta_{n}$ as a zero of $P$ exceeds the number of the roots of $P$ lying away from the unit circle (see \autoref{lem:NumRoots}), which is a contradiction.

%\textcolor{red}{Discuss the new ideas and contributions here. Explain that we include the all the details for the purpose of completeness and clarity of the modifications in the new settings, and by no means we attempt to claim the originality of the arguments.}

\subsection{Structure of the article}

In Section \ref{sec:DistRoots} we collect some results about  roots of  polynomials with integer or algebraic coefficients, which we will use at various places in the paper, in particular in the proofs of Theorems \ref{thm:main} and \ref{thm:EAA}.  In Section \ref{sec:entropies}, we review two notions of entropy and  present their basic properties.  In Section \ref{sec:pf-main} we prove  Theorem \ref{thm:main} by assuming Theorems \ref{thm:EAA} and \ref{thm:UBGarsia}.  The proofs of Theorems \ref{thm:EAA} and \ref{thm:UBGarsia} are respectively given in Sections \ref{sec:pf-EAA} and \ref{sec:UBGarsia}.

\section{Distribution of roots of polynomials}\label{sec:DistRoots}

In this section, we collect some properties about  roots of polynomials  with  coefficients in a given finite set of algebraic numbers,  which we will mainly use in the proofs of Theorems \ref{thm:main} and \ref{thm:EAA}. Many of them are standard, and the others are the extensions of  the corresponding results in  \cite{BreuillardVarju2019, RapaportVarju2024} about integer polynomials.

Throughout this section, let $\mathcal D$ be a fixed finite set of algebraic numbers in $\bbC$. For $n\geq 1$,  let $\mathcal P^n$ denote the collection of polynomials in $X$ of degree not exceeding $n$ and with coefficients in $\mathcal D$. That is,
\begin{equation}\label{eq:def-calPn1}
	\calP^{n} = \left \{ \sum_{i=0}^{n} a_{i}X^{i} \colon  a_{i}\in \calD \mbox{ for all } 0\leq i\leq n \right \}.
\end{equation}

\subsection{Polynomials, Mahler measure and height}

%Throughout the paper, we use $ \bbZ $,  $\bbR$,  $\bbQ$,   and $ \bbC$ to denote the sets of integers, real numbers,  rational numbers,    and complex numbers, respectively.

We first introduce  some standard notation and definitions. For   $\calS \subset \bbC$, by $\calS[X]$ and $ \calS[X_{1}, \ldots, X_{k}]$ we respectively denote the sets of  the univariate polynomials  in variable $X$ and the multivariate polynomials in variables $X_{1}, \ldots, X_{k} $ with coefficients in $\calS$.

For $ \alpha, \alpha_{1}, \ldots, \alpha_{k} \in \bbC $, we write $$ \calS[\alpha] = \left \{ P(\alpha)\colon P\in \calS[X] \right \}$$ and $$ \calS[\alpha_{1}, \ldots, \alpha_{k}] = \left \{ P(\alpha_{1}, \ldots, \alpha_{k}) \colon P \in \calS[X_{1}, \ldots, X_{k}]\right \},$$ where $P(\alpha) $ or $P(\alpha_{1}, \ldots,\alpha_{k})$ denotes the evaluation of a polynomial $P$ at $ \alpha$ or $(\alpha_{1}, \ldots, \alpha_{k})$. Write $ \calS[\calA] = \calS[\alpha_{1}, \ldots, \alpha_{k}] $ if $ \calA = \{\alpha_{1}, \ldots, \alpha_{k}\} $. We adapt the notation with $ \calS $ replaced by $ \bbQ, \bbZ $,  $\bbC$ or any field $\bbK \subset \bbC$. For $ \calA \subset \bbC $, let $ \bbQ(\calA) $ denote the smallest field containing $ \bbQ $ and $ \calA $, and let $ \bbQ(\alpha) $ be the smallest field containing  $ \bbQ $ and $ \alpha \in \bbC $.

For $ f = \sum_{i=0}^{n} a_{i}X^{i} \in \bbC[X] $, the degree of $ f $ is denoted as $ \deg f $, and we write
\begin{equation*}
	\ell_q(f) = \left (\sum_{i=0}^{n}\abs{a_{i}}^{q} \right )^{1/q} \; \mbox{ for }1\leq q<\infty,\quad \mbox{ and }\;\ell_{\infty}(f) = \max_{0\leq i\leq n} \abs{a_{i}}.
\end{equation*}
    Similar notations are used for multivariate polynomials.
\begin{definition}[Mahler measure]\label{def:MahlerMeasure}
	The \textit{Mahler measure} of $ f = \sum_{i=0}^{n} a_{i}X^{i} \in \bbC[X] $ is defined as
	\begin{equation}\label{eq:def-Mahler}
		M(f) := \abs{a_{n}} \prod_{i=1}^{n} \max\left\{1, \abs{\alpha_{i}} \right\},
	\end{equation}
	where $ \alpha_{1}, \ldots, \alpha_{n} $ are the roots of $ f $. For an algebraic number $ \alpha\in \bbC$, the Mahler measure of $ \alpha$ is defined by
	$ M(\alpha) := M(g) $, where $ g $ is the minimal polynomial of $ \alpha $ in $ \bbZ[X] $ with coprime coefficients.
\end{definition}
By the above definition, $M(\alpha) \geq 1$ for every  algebraic number $ \alpha$. Following \cite[Chapter 14]{Masser2016}, we define the \textit{height}  of an algebraic number $ \alpha \in \bbC$ as
\begin{equation}
\label{e-height}
	H(\alpha) = M(\alpha)^{1/\deg \alpha},
\end{equation}
where $ \deg \alpha $ stands for  the degree of  $ \alpha $ over $ \bbQ $.  Clearly $H(\alpha)\geq 1$.

It is immediate from \eqref{eq:def-Mahler} that
\begin{equation*}\label{eq:MahlerMulti}
	M(fg) = M(f)M(g) \mFor f, g \in \bbC[X].
\end{equation*}
Below we list  more properties about the Mahler measure and height.
\begin{lemma}[{\cite[Lemma 1.6.7]{BombieriGubler2006}}]\label{lem:MahlerNorms}
	Let $ f \in \bbC[X] $ with $ \deg f = n $. Then $ M(f) \leq \ell_1(f) $. Moreover,
	\begin{equation*}\label{eq:MahlerNorms}
		\binom{n}{\fl{n/2}}^{-1} \ell_\infty(f) \leq M(f) \leq \ell_2(f) \leq \sqrt{n+1} \, \ell_\infty(f),
	\end{equation*}
where 	$\fl{x}$ stands for the integral part of $x$.
\end{lemma}
\begin{lemma}[{\cite[Equation 4]{Mahler1960}}]\label{lem:Mahler-LB}
	Let $ f \in \bbC[X] $ with $ \deg f = n $. Then
	\begin{equation*}\label{eq:NormMahler}
		\ell_1(f)\leq 2^{n} M(f).
	\end{equation*}
\end{lemma}

\begin{lemma}[{\cite[Proposition 14.7]{Masser2016}}] \label{lem:PolyHeights}
	Let $ f \in \bbZ[X_{1}, \ldots, X_{k}] $ such that the partial degree of $ f $ in $ X_{j} $ is at most  $ L_{j} $ for $ 1 \leq j \leq k $. Let $ \alpha_{1}, \ldots, \alpha_{k} $ be algebraic numbers. Then
	\begin{equation*}\label{eq:PolyHeights}
		H(f(\alpha_{1}, \ldots, \alpha_{k})) \leq \ell_1(f) \prod_{j=1}^{k} H(\alpha_{j})^{L_{j}}.
	\end{equation*}
\end{lemma}

Let $ [\bbK_{1} : \bbK_{2}]$ denote the degree of a field extension $ \bbK_{2} / \bbK_{1} $, that is, the dimension of $ \bbK_{1}$ as a vector space over $\bbK_{2}$ (\cite{Morandi1996}).

\begin{lemma}[{\cite[Proposition 14.13]{Masser2016}}] \label{lem:PolyVal-LB}
	Let $ \bbK \subset \bbC $ be an extension field of $ \bbQ $ with $ k = [\bbK : \bbQ ] < \infty $. If $ \alpha \neq 0 $  is in $\bbK$, then $ \abs{\alpha} \geq H(\alpha)^{-k} $.
\end{lemma}

\begin{lemma}[{\cite[Corollary 5.8]{Morandi1996}}]
\label{lem-extension}
Let $\alpha_1,\ldots, \alpha_k$ be algebraic numbers. Then there exists an algebraic integer $\theta$ such that $\bbQ(\alpha_1,\ldots, \alpha_k)=\bbQ(\theta)$.
\end{lemma}

\subsection{The distribution of roots of polynomials with coefficients in $\mathcal D$}

   Recall that $\mathcal D$ is a fixed finite set of algebraic numbers, and $\mathcal P^n$, $n\in \bbN$, are defined as in \eqref{eq:def-calPn1}.

The following proposition 
 is a slight extension of  \cite[Lemma 4.2]{RapaportVarju2024}  (see also, \cite{BEAUCOUPEtAl1998} and \cite[Lemma 26]{BreuillardVarju2019}).

\begin{proposition}\label{lem:NumRoots}
	For $ \varepsilon>0 $, there exists $ k = k(\varepsilon,\calD) > 0 $ such that every polynomial $ P $ in $\Union_{n=1}^{\infty} \calP^{n} $ has at most $ k $ nonzero roots of modulus at most $ 1-\varepsilon $.
\end{proposition}

\begin{proof}
The proposition was proved  in \cite[Lemma 4.2]{RapaportVarju2024} in the case when $\mathcal D$ is a finite set of integers.  Here we follow that proof  with slight modifications. Let $\varepsilon>0$. Set
$$\rho=\frac{\max\{|s|:\; 0\neq s\in \mathcal D\}}{\min\{|s|:\; 0\neq s\in \mathcal D\}}.$$
 Define
$$
a(j) = \frac{j}{j+1}\cdot\frac{1}{[\rho(j+1)]^{1/j}} \quad \mbox{ for } j\in \bbN.
$$
Then $ a(j)\to 1 $ as $ j\to \infty $. So we may pick a large integer $k$ such that $a(k)>1-\varepsilon/2$.

Let $P\in \mathcal P^n$ for some $n\in \bbN$, and let $ z_{1}, \ldots, z_{N} $ be the nonzero roots, repeated according to multiplicity,  of $ P $ in the open disc $ \abs{z} <  1-\varepsilon $.  Below we prove that $N\leq k$.

Set $u=\min\{|s|:\; 0\neq s\in \mathcal D\}$ and $$Q= \frac{P}{u X^{m}},$$  where $ m $ is the lowest degree of the monomials in $ P $. Then  $ \abs{Q(0)} \geq 1 $ and $ \ell_\infty(Q) \leq \rho $. Write $r=k/(k+1)$. Then $1-\varepsilon/2<a(k)<r$.  By Jensen's formula and since  $\abs{Q(0)} \geq 1$,
	\begin{equation}
\label{e-Jensen} \sum_{j=1}^{N} \log\frac{r}{\abs{z_{j}}} \leq \int_{0}^{1}\log|Q(re^{2\pi i t})|dt.
\end{equation}
	Note that for each $t\in [0,1]$, $$ \abs{Q(re^{2\pi i t})} \leq \rho (1+r+r^{2}+\cdots)= \frac{\rho}{1-r}=\rho (k+1).$$
 Since $ \abs{z_{j}} \leq a(k) $ for $ 1 \leq j \leq N $, by \eqref{e-Jensen},
	\begin{equation*}
		N \log (\rho(k+1))^{1/k} \leq \log (\rho(k+1)).
	\end{equation*}
	This concludes $ N \leq k $.
\end{proof}

As a consequence of the above proposition, we give the following corollary which  is an analogue of \cite[Lemma 4.3]{RapaportVarju2024} in our setting.

\begin{corollary}
\label{cor-temp1}
For any $\varepsilon\in (0,1/2)$, there exists $c=c(\varepsilon)\in (0,1)$ such that the following holds. Let $n\in \bbN$, $0\neq P\in {\mathcal P}^n$ and $0<r<\varepsilon^n 2^{-n}$. Suppose that $|P(\lambda)|\leq r$ for some $\lambda\in \bbC$ with $\varepsilon\leq |\lambda|\leq 1-\varepsilon$. Then there exists $\eta\in \bbC$ such that $P(\eta)=0$ and
$$
|\lambda-\eta|\leq (2^n\epsilon^{-n}r)^c.
$$
\end{corollary}
\begin{proof}
The proof is nearly identical to that of  \cite[Lemma 4.3]{RapaportVarju2024}. For the reader's convenience, we include the details.

 Let $\varepsilon\in (0,1/2)$, and let $n, P, r$ be given as in the statement of the lemma. By Proposition \ref{lem:NumRoots}, there exists an integer $k=k(\epsilon,\mathcal D)$ (independent of $P$) such that $P$ has at most $k$ nonzero roots of modulus at most $1-\varepsilon/2$ (repeated according to multiplicity).  Denote these roots by $\eta_1,\ldots, \eta_m$. Then $m\leq k$ and
$$
r\geq |P(\lambda)|\geq (\varepsilon/2)^{n-m}\cdot \prod_{j=1}^m |\eta_j-\lambda|.
$$
Hence there exists some $1\leq j\leq m$ such that
$$
|\eta_j-\lambda|\leq (r\cdot (\varepsilon/2)^{-n})^{1/m}\leq (2^n\epsilon^{-n}r)^{1/k}.
$$
This completes the proof of the corollary by taking $c=1/k$.
 \end{proof}

\subsection{A separation property of roots of polynomials in $\mathcal P^n$}
\label{subsec:SepRoots} The main result in the subsection is the following proposition, which is an analogue of \cite[Lemma 4.1]{RapaportVarju2024}.

\begin{proposition}\label{prop:SepRoots}
	There exists $ M = M(\calD) > 0 $ such that the following holds for all sufficiently large $ n $. Let $ \lambda_{1} \neq \lambda_{2} $ be two algebraic numbers each of which is a root of a polynomial in $ \calP^{n} $. Then $ |\lambda_{1}-\lambda_{2}|>2n^{-Mn} $.
\end{proposition}

To prove the above proposition,  we start with a result essentially due to Mahler~\cite{Mahler1964}.
For $ n \in \bbN $ and  $ a> 0 $, write
\begin{equation}
\label{e-Fna}
	\calF(n, a) := \left\{ P\in \bbZ[X] : \ell_\infty(P) \leq a \text{ and } \deg P \leq n \right\}.
\end{equation}

\begin{lemma}\label{thm:Mahler}
Let $n\geq 4$, $a>0$, and  let $ \lambda_{1}, \lambda_{2} $ be two distinct algebraic numbers each of which is a root of a polynomial in $ \calF(n, a) $. Then $$ |\lambda_{1}-\lambda_{2}|>2n^{-4n} a^{-4n+2}.$$
\end{lemma}

\begin{proof}
We follow the proof of \cite[Lemma 4.1]{RapaportVarju2024} with minor modifications. Let $n, a, \lambda_1,\lambda_2$ be given as in the statement of the lemma.  	Suppose $ P_{1}(\lambda_{1}) =0$ and $ P_{2}(\lambda_{2}) =0 $ for some $ P_{1}, P_{2} \in \calF(n, a) $. Let $ P_{1}=R_{1}^{k_{1}}\cdots R_{s}^{k_{s}} $ and  $ P_{2}=Q_{1}^{h_{1}}\cdots Q_{t}^{h_{t}} $ be respectively the irreducible factorizations of $ P_{1} $ and $ P_{2} $ in $ \bbZ[X] $.
	
	Without loss of generality assume that $ R_{1}(\lambda_{1})=0 $ and $ Q_{1}(\lambda_{2})=0 $. If $\lambda_{1}$ and $\lambda_{2}$ are Galois conjugates, then take $ P= R_{1} $ or $ Q_{1} $, otherwise, take $ P=R_{1}Q_{1} $. In both cases,  $ \deg P\leq 2n $, and  $ M(P)\leq M(P_{1})M(P_{2})\leq (n+1)a^{2} $ by Lemma \ref{lem:MahlerNorms}.
%Moreover, the discriminant $ D(P) $ of $ P $ satisfies $ |D(P)|\geq 1 $ since $ P\in \bbZ[X] $ is square-free.
Hence applying \cite[Theorem 2]{Mahler1964} to $ P $ gives
	\begin{align*}
		\abs{\lambda_{1}-\lambda_{2}} &> \sqrt{3}(2n)^{-(n+1)}((n+1)a^{2})^{-2n+1} \\
 &>\sqrt{3}\, 2^{-3n} n^{-3n}  a^{-4n+2}\\
&>2 n^{-4n} a^{-4n+2},
	\end{align*}
	where we use the assumption   $ n \geq 4$ in the last inequality.
\end{proof}

To apply Lemma \ref{thm:Mahler}, we need the following result.

\begin{lemma}\label{lem:ToUseMahler} Let $\lambda\in \bbC$.
	If $\lambda$ is a root of a polynomial $ f \in \calP^{n} $ for some $n\geq 1$, then $\lambda$ is a root of a polynomial $ F \in \calF(d n, C(n+1)^{d} ) $, where $d=[\bbQ({\mathcal D}): \bbQ]$, $C$ is a constant depending on ${\mathcal D}$ and $\calF(\cdot, \cdot)$ is defined as in \eqref{e-Fna}.
\end{lemma}

\begin{proof}
	By Lemma \ref{lem-extension}, we can choose an algebraic integer $\theta$ such that $\bbQ(\theta)=\bbQ({\mathcal D})$. Then we can take $L\in \bbN$ such that
\begin{equation}
\label{e-integer}
Ls \in \bbZ[\theta] \quad \mbox{ for all }s\in \mathcal D.
\end{equation}

Notice that the degree of $\bbQ(\mathcal D)$ over $\bbQ$ is $d$, so $\deg(\theta)=d$. Let $\theta_1,\ldots, \theta_d$ be the algebraic conjugates of $\theta$, with $\theta_1=\theta$. For $j=1,\ldots, d$, let $$\sigma_j: \bbQ(\theta)\to \bbQ(\theta_j)$$ be the field isomorphism such that $\sigma_j(a)=a$ for all $a\in \bbQ$ and $\sigma_j(\theta)=\theta_j$.

Let $\lambda\in \bbC$ and suppose that $P(\lambda)=0$ for some $P=\sum_{i=0}^na_i X^i\in \mathcal P^n$.  Define $F\in \bbC[X]$ by
$$
F=L^d \prod_{j=1}^d \left(\sum_{i=0}^n \sigma_j(a_i) X^i\right)= \prod_{j=1}^d \left(\sum_{i=0}^n \sigma_j(La_i) X^i\right).
$$
Clearly $F(\lambda)=0$, $\deg(F)\leq dn$ and  $$\ell_\infty(F)\leq \ell_1(F)\leq L^d \prod_{j=1}^d \left(\sum_{i=0}^n |\sigma_j(a_i)|\right)\leq C (n+1)^d,$$ where
$
C:=L^d \max\{|\sigma_j(s)|^d:\; 1\leq j\leq d, \;s\in \mathcal D\}.
$

To conclude the lemma, it remains to show that $F\in \bbZ[X]$. By \eqref{e-integer}, $La_i\in \bbZ[\theta]$ for all $0\leq i\leq n$.
Hence there exist integer polynomials $Q_i$, $i=0,\ldots, n$, such that $La_i=Q_i(\theta)$. It follows that for  $0\leq i\leq n$ and $1\leq j\leq d$, $$\sigma_j(La_i)=\sigma_j(Q_i(\theta))=Q_i(\theta_j). $$
Therefore $$F =\prod_{j=1}^d \left(\sum_{i=0}^n Q_i(\theta_j) X^i\right)=\sum_{k=0}^{dn}\beta_kX^k.$$
Notice that  each $ \beta_{k} $ is a symmetric  polynomial in $ \theta_{1}, \ldots, \theta_{d} $ with integral coefficients. Since  $ \theta  $ is an algebraic integer, it follows that  $ \beta_{k} \in \bbZ $, and so $ F \in \bbZ[X] $.
\end{proof}

Now we are ready to prove Proposition \ref{prop:SepRoots}.

\begin{proof}[Proof of Proposition \ref{prop:SepRoots}]
Let $d, C$ be the constants as in Lemma \ref{lem:ToUseMahler}.
	A combination of Lemmas \ref{thm:Mahler} and \ref{lem:ToUseMahler} gives
	\begin{equation*}
		\abs{\lambda_{1}-\lambda_{2}}  >2(d n)^{-4 d n} [(C (n+1))^{d}]^{-4d n+2}>2n^{-(4d^{2}+4d+1)n},
	\end{equation*}
	for  $ n $ sufficiently large depending on $ \calD $. Taking $ M = 4 d^{2} + 4d + 1 $ finishes the proof.
\end{proof}

\subsection{The order of zeros of  polynomials under certain constrains} \label{subsec:multiplicity}

The following proposition is an analogue of \cite[Lemma 2.3]{RapaportVarju2024} in our setting.
\begin{proposition}\label{prop:algLargeOrder}
	 Let $ \varepsilon > 0 $ and $k\in {\bbN}$. There exists a positive integer $N=N(\epsilon, k, \mathcal D)$ such that the following holds.  Let $n\in \bbN$ with $n\geq N$, and $ \lambda,\eta \in (-1, 1) $ such that $ \eta $ is algebraic with $\deg \eta \leq n $ and $ 0 < \abs{\lambda-\eta} \leq 1 $. Let $ n'\in \bbN $ such that $ n' \geq (k+2) n^{1+\varepsilon}$. Let $ 0 \neq P \in \calP^{n'} $. Suppose that
	\begin{equation}\label{eq:Cond-AlgOrdLarge}
		(2M(\eta))^{dn'/k}\abs{P(\lambda)}^{1/k} \leq \abs{\lambda - \eta} \leq (2M(\eta))^{-dn'},
	\end{equation}
	where $d:=[\bbQ(\mathcal D):\bbQ]$. Then $ \eta $ is a zero of $ P $ of order at least $ k $.
\end{proposition}

The proof of  Proposition \ref{prop:algLargeOrder} relies on the following result  which is contained in the proof of \cite[Lemma 4.6]{RapaportVarju2024}. As was noted by Rapaport and Varj\'u \cite{RapaportVarju2024}, this result was due to V. Dimitrov.

\begin{lemma}\label{lem:Val4Derivative}
	Let $ \lambda, \eta$ be  distinct nonzero numbers in $(-1,1)$. Let $0\neq P \in \calP^{n'}$ for some $n'\in \bbN$. Let $ m $ be the order of vanishing of $P$  at $\eta$, where we allow $m = 0$. Then
	\begin{equation*}
		\Abs{ \frac{P^{(m)}(\eta)}{m!}} \leq (n')^{m+2}D\abs{\lambda - \eta} + \dfrac{\abs{P(\lambda)}}{\abs{\lambda -\eta}^{m}},
	\end{equation*}
	where $ P^{(i)}$ denotes the $i$-th derivative of $ P $ for $ i \in \bbN$, and $ D:=\max \{|s|:\; s\in \mathcal D\} $.
\end{lemma}

%\begin{lemma}\label{lem:toProveAlgMulti}
%	Let $ k, n \in \bbN $ and $  \lambda \in (-1, 1) $. Let $ \eta \in \Alg \intxn (-1, 1) $ with $ \deg \eta \leq n $ and $ 0 < \abs{\lambda-\eta} \leq 1 $. Let $ n' \in \bbN $ with $ n' > 1/D $. Let $ 0 \neq P \in \calP^{n'} $. Let $ \alpha > 0 $ such that
%	\begin{equation}\label{eq:cond_alpha}
%		\log\alpha > \frac{ [kn + (k+2)/C_{1}]\log n' + n\log(n'+1) + n\log C_{2}+ (\log(2D))/C_{1} }{ n' } .
%	\end{equation}
%	Suppose
%	\begin{equation}\label{eq:alpha_cond_dist}
%		(\alpha M(\eta))^{C_{1}n'/k}\abs{P(\lambda)}^{1/k} \leq \abs{\lambda - \eta} \leq (\alpha M(\eta))^{-C_{1}n'}.
%	\end{equation}
%	Then $ \eta $ is a zero of $ P $ of order at least $ k $.
%\end{lemma}

\begin{proof}
%	By \eqref{eq:alpha_cond_dist},
%	\begin{equation}\label{eq:dist-CondSum}
%		\abs{\lambda - \eta} + \dfrac{\abs{P(\lambda)}}{\abs{\lambda -\eta}^{k}} \leq 2 (\alpha M(\eta))^{-C_{1}n'}.
%	\end{equation}
%	
%	Let $ m $ be the order of vanishing of $ P $ at $ \eta $. We allow $ m = 0 $. Suppose on the contrary that $ m < k $.
	
We follow the proof of \cite[Lemma 4.6]{RapaportVarju2024}.	Write $ P = \sum_{j=0}^{n'} a_{j} X^{j} $ with $ a_{j} \in \calD$.
	Denote
	\begin{equation}\label{eq:ExpressQ}
		Q = \frac{P^{(m)}}{m !} = \sum_{j=m}^{n'} \binom{j}{m} a_{j} X^{j-m}.
	\end{equation}
	Then $ Q(\eta) \neq 0 $ by the definition of $ m $. By Taylor's theorem with Lagrange remainder term, there exists some $ \xi $ between $ \lambda $ and $ \eta $ such that
	\begin{equation}\label{eq:Taylor}
		P(\lambda) = Q(\xi) (\lambda - \eta)^{m}
	\end{equation}
	If $ m = 0 $, this holds with $ \xi = \lambda $.
	
	It follows from \eqref{eq:ExpressQ} that $ \deg Q \leq n' $ and $ \ell_\infty(Q) \leq (n')^{m} D $. Then
	\begin{equation*}
		\abs{Q^{(1)}(x)}   \leq (n')^2\ell_\infty(Q)\leq  (n')^{m + 2 } D \mFor  x \in (-1, 1).
	\end{equation*}
	Hence by the mean value theorem,
	\begin{equation*}
		\abs{Q(\eta)} \leq (n')^{m +2}D\abs{\xi - \eta} + \abs{Q(\xi)}.
	\end{equation*}
	From this, $ \abs{\xi - \eta} \leq \abs{\lambda-\eta} $ and   \eqref{eq:Taylor}, we obtain that
	\begin{equation*}\label{eq:PolyVal-UB}
		\abs{Q(\eta)} \leq (n')^{m+2}D\abs{\lambda - \eta} + \dfrac{\abs{P(\lambda)}}{\abs{\lambda -\eta}^{m}},
	\end{equation*}
	which finishes the proof.
\end{proof}

Now we are ready to prove Proposition \ref{prop:algLargeOrder}.

\begin{proof}[Proof of Proposition \ref{prop:algLargeOrder}]
	%Recall $ C_{1}$ and $C_{2} $ from \eqref{eq:C1C2}.
Our argument is adapted from the proof of  \cite[Lemma 4.6]{RapaportVarju2024}.	
Let $\varepsilon, k, \lambda, \eta, n', P$ be given as in the statement of the proposition.  Let $m$ be the order of vanishing of $P$ at $\eta$. We show below that $m\geq k$.

Write $Q = P^{(m)}/(m!) $.  Then $ Q(\eta) \neq 0$.
		From \eqref{eq:ExpressQ} we see that $$Q(\eta)=\widetilde{Q} (s_1,\ldots, s_r, \eta) $$ for an integer  multivariate polynomial
		$\widetilde{Q}$ with  $\ell_1(\widetilde{Q}) \leq (n'+1)(n')^{m} $, where $s_1,\ldots, s_r$ are the elements in $\mathcal D^+$.
	By Lemma \ref{lem:PolyHeights}, the height of $\widetilde{Q}(\eta)$ (cf. \eqref{e-height}) satisfies
	\begin{equation}
	\label{e-HQ}
	\begin{split}
		H(Q(\eta)) &=H\left(\widetilde{Q} (s_1,\ldots, s_r, \eta)\right)\\
		&\leq \ell_1(\widetilde{Q}) \left( \prod_{s\in \calD^{+}} H(s) \right)\cdot H(\eta)^{n'}\\
		&\leq
		(n'+1)(n')^{m} C_{2} M(\eta)^{n'/\deg \eta},
	\end{split}
	\end{equation}
where
$C_{2}:=\prod_{s\in \calD^{+}} H(s)$.
	
Note that by \eqref{eq:ExpressQ}, $Q(\eta)\in \bbQ(\mathcal D, \eta)$.  It follows from  \cite[Propositions 1.15 and 1.21]{Morandi1996} that
\begin{equation}
\label{e-Qeta}
 \deg Q(\eta)=[\bbQ(Q(\eta)):\bbQ]\leq [\bbQ(\mathcal D,\eta): \bbQ]\leq [\bbQ(\mathcal D):\bbQ][\bbQ(\eta):\bbQ]=d \deg(\eta),
 \end{equation}
where $d:=[\bbQ(\mathcal D):\bbQ]$.
Since  $ Q(\eta) \neq 0 $,  $\deg(\eta)\leq n$ and $H(Q(\eta))\geq 1$, it follows from   Lemma \ref{lem:PolyVal-LB} , \eqref{e-Qeta} and \eqref{e-HQ} that
	\begin{equation}\label{eq:PolyVal-LB}
		\abs{Q(\eta)}\geq H(Q(\eta))^{-\deg Q(\eta)} \geq H(Q(\eta))^{-d \deg\eta} \geq \left((n'+1)(n')^{m}C_{2}\right)^{-dn} M(\eta)^{-dn'}.
	\end{equation}

	Suppose on the contrary that $ m < k $. Set $D=\max \{|s|:\; s\in \mathcal D\}$. Then by Lemma \ref{lem:Val4Derivative} and $ \abs{\lambda-\eta} \leq 1 $,
	\begin{align*}
		 |Q(\eta)|=\frac{|P^{(m)}(\eta)|}{m !}&\leq (n')^{m+2}D\abs{\lambda - \eta} + \dfrac{\abs{P(\lambda)}}{\abs{\lambda -\eta}^{m}}\\
&\leq (n')^{k+2}D\abs{\lambda - \eta} + \dfrac{\abs{P(\lambda)}}{\abs{\lambda -\eta}^{k}}.
	\end{align*}
 	Hence
 \begin{align*}
		&\left((n'+1)(n')^{k}C_{2}\right)^{-dn}M(\eta)^{-dn'}\\
&\qquad\mbox{}\leq \abs{Q(\eta)} & \mbox{(by \eqref{eq:PolyVal-LB} and $m<k$)}\\
&\qquad \mbox{}\leq (n')^{k+2}D\abs{\lambda - \eta} + \dfrac{\abs{P(\lambda)}}{\abs{\lambda -\eta}^{k}}\\
&\qquad\mbox{}\leq \left((n')^{k+2}D +1 \right)(2M(\eta))^{-dn'} & \mbox{(by \eqref{eq:Cond-AlgOrdLarge})},
	\end{align*}
	which is simplified to
	\begin{equation*}
		\left((n')^{k+2}D +1 \right) \left((n'+1)(n')^{k}C_{2}\right)^{d n} \geq  2^{d n'}.
	\end{equation*}
	 This contradicts the assumption that $n' \geq (k+2) n^{1+\varepsilon}$, when $ n $ is large enough depending on  $\calD$, $k$ and $\varepsilon$.
\end{proof}

%\begin{proof}[Proof of \autoref{prop:algLargeOrder}]
%	Since $ n' \geq (k+2) n^{1+\varepsilon}  $ for some $ \varepsilon > 0 $, the right hand side of \eqref{eq:cond_alpha} is smaller than $ 1/2 $ when $ n $ is sufficiently large depending on $ \calD $. Then $ \alpha = 2 $ satisfies \eqref{eq:cond_alpha}. Hence the proof is completed by  \autoref{lem:toProveAlgMulti}.
%\end{proof}

\section{Preliminaries on Shannon and differential entropies}\label{sec:entropies}

In this section, we  briefly review two notions of entropy, and present some of their properties which will be used in the later sections.  For more background material on entropy, the reader is referred to \cite{BreuillardVarju2020,CoverThomas2006}.

Let $ X $ be a random variable in $\euclid[d]$.  If $ X $ is a discrete random variable taking  values in a countable set $ \{x_{j}\}_{j} $, the \textit{Shannon entropy} of $ X $ is defined as
\begin{equation}\label{eq:def-Shannon}
	H(X) := \sum_{j} - \bbP \{ X = x_{j} \} \log \bbP \{ X = x_{j} \}.
\end{equation}
If $ X $ is an absolutely continuous random variable with density $ f \colon \euclid[d] \to [0, \infty) $, the \textit{differential entropy} of $ X $ is  defined as
\begin{equation}\label{eq:def-DiffEntropy}
	H(X) := \int_{\bbR^d} - f(x) \log f(x) \, dx.
\end{equation}
 (Recall the definition of $ H(p) $  for a probability vector $p$ (see \eqref{eq:def-DynQuantities}) and the height $H(\alpha)$  for an algebraic number $\alpha$ (see \eqref{e-height}). The multiple use of $ H(\cdot) $ should cause no confusion, as it will always be clear from the context which type of the input is.)

Let $ A $ be an invertible $d\times d$ real matrix and $ b \in \euclid[d] $. If $ X $ is a discrete random variable,  it is easy to see that
\begin{equation*}\label{eq:Affine-ShannonEntropy}
	H(A X + b) = H(X).
\end{equation*}
If $ X $ is an absolutely continuous random variable with finite differential entropy, then  it follows from the change of variables formula that
\begin{equation}\label{eq:Affine-DiffEntropy}
	H(AX + b)  = H(X) + \log \Abs{\det A}.
\end{equation}

%Let $ \calL $ denote the Lebesgue measure and write $ \mu \ll \calL $ if $ \mu $ is absolutely continuous to $ \calL $. Define
%\begin{equation}\label{eq:def-ShannonDiff}
%	H(X) = \begin{cases}
%		\sum_{x \in \bbR } - \bbP \{ X = x \} \log \bbP \{ X = x \} & \text{if $ \mu $ is discrete}; \\
%		\int_{\bbR} - f(x) \log f(x) \, dx & \text{if $ \mu \ll \calL $ with density $ f $.}
%	\end{cases}
%\end{equation}
%By convention we write $ H(\mu) = H(X) $.

\section{Proof of Theorem \ref{thm:main} }
\label{sec:pf-main}

In this section, we prove Theorem \ref{thm:main} by assuming Theorems \ref{thm:EAA} and \ref{thm:UBGarsia}. The proofs of Theorems \ref{thm:EAA} and \ref{thm:UBGarsia} will be given in the next two sections respectively.

Throughout this section, let $t_1,\ldots, t_m$ be real algebraic numbers and $p=(p_1,\ldots, p_m)$  a probability vector. Let $\nu$ denote the atomic probability measure $\sum_{j=1}^m p_j\delta_{t_j}$.
  For any $0\neq \eta\in (-1,1)$, let $\mu_\eta$ denote the self-similar measure associated with the IFS $\{\eta x+t_j\}_{j=1}^m$ and $p$, and let $h_{\eta, \nu}$ denote the Garsia entropy of $\eta$ and $\nu$ (see \eqref{e-Garsia}).  Let $\mathcal P^n$, $n\in \bbN$, be  defined as in \eqref{eq:def-calPn}.

  The proof of Theorem \ref{thm:main} also relies  on the following two results of Hochman.

  \begin{theorem}
 [{\cite{Hochman2014}}]
  \label{thm:SuperExpClose}
	Let $0\neq  \lambda\in (-1, 1)$. Suppose that $$ \dim \mu_{\lambda} < \min \left\{1, \;\frac{ H(p)}{- \log \abs{\lambda}}\right\} .$$ Then for every $ \delta \in (0,1) $ and for all sufficiently large $ n $ (depending on $ \lambda $ and $ \delta $), there is a polynomial $ 0 \neq P \in \calP^{n} $ such that $ \abs{P(\lambda)} < \delta^{n} $.
\end{theorem}

\begin{theorem} [{\cite{Hochman2014}}]
\label{thm:DimFormula}  For each  nonzero algebraic number $\eta\in (-1, 1)$,
	$$	
		\dim \mu_{\eta} = \min \left \{ 1,\; \frac{h_{\eta,\nu}}{- \log \abs{\eta} }\right \}.
	$$
If in addition the IFS $\{\eta x+t_j\}_{j=1}^m$ has no exact overlaps, then
$$	
		\dim \mu_{\eta} = \min \left \{ 1,\; \frac{H(p)}{- \log \abs{\eta} }\right \}.
	$$	
\end{theorem}

Now we are ready to prove Theorem \ref{thm:main}.

\begin{proof}[Proof of Theorem \ref{thm:main}]  We follow the proof of \cite[Theorem A.8]{RapaportVarju2024} with slight modifications. For the reader's convenience, we provide the full details.

By Theorem \ref{thm:DimFormula}, we may assume that $ \lambda$ is transcendental. Since all $t_j$ are algebraic numbers,   the IFS $ \{\lambda x+t_j\}_{j=1}^m $ has no exact overlaps. Suppose on the contrary that
\begin{equation*}
	\dim \mu_{\lambda} < \min \left\{1,\; \frac{H(p)}{- \log \abs{\lambda} } \right\}.
\end{equation*}
Let $ \varepsilon >0 $ such that
\begin{equation}
\label{e-t1}
	\dim \mu_{\lambda} + 2\varepsilon < \min\left\{1, \;\frac{H(p)}{ - \log \abs{\lambda} }\right\}.
\end{equation}
Let $ M > 0 $ be large with respect to $t_1,\ldots, t_m, \lambda, p $ and $ \varepsilon $. Let $ n_{0} \in \bbN $ be large depending on $ M $.

By Theorem \ref{thm:EAA}, there exist  an integer $n>n_0$  and a nonzero algebraic number $\eta$ in $(-1,1)$ such that  $\eta$ is a  root of some polynomial in $ \calP^{n} $ satisfying
\begin{equation}\label{eq:eta-property}
	\dim \mu_{\eta} < \dim \mu_{\lambda} + \varepsilon \mAnd \abs{\lambda - \eta} < \exp(-n^3).
\end{equation}
We can assume $ n_{0} $ is large enough so that $ \abs{\eta} < (1+\abs{\lambda})/2 $.
 By \eqref{e-t1} and \eqref{eq:eta-property},
\begin{equation}
\label{e-t2}
\dim \mu_{\eta}< \min\left\{1, \;\frac{H(p)}{ - \log \abs{\lambda} }\right\}-\epsilon.
\end{equation}

Since $\eta$ is algebraic, it follows from  Theorem \ref{thm:DimFormula} that
	\begin{equation*}
		\dim \mu_{\eta} = - \frac{h_{\eta,\nu}}{\log\abs{\eta}}.
	\end{equation*}
From this, \eqref{e-t2} and Theorem \ref{thm:UBGarsia}, we may assume that $M(\eta)<M$.

%Suppose on the contrary that there is a sequence $ (n_{k}) $ such that $ M(\eta_{n_{k}}) \to \infty $ as $ k \to\infty $. It follows from By \autoref{thm:UBGarsia} and $ \abs{\lambda - \eta_{n} } < \varepsilon_{n} $ that $ \dim \mu_{\eta_{n_{k}}} \to - H(p)/\log\abs{\lambda} $ as $ k \to \infty $. This contradicts $ \dim \mu_{\eta_{n}} \leq \dim \mu_{\lambda} + \varepsilon \leq - H(p)/\log \abs{\lambda} - \varepsilon $.

Set $d= [\bbQ(\calD) : \bbQ]$. Since $ Q(\eta) = 0 $ for some $ Q \in \calP^{n} $, it follows that
$$[\bbQ(\mathcal D, \eta):\bbQ(\mathcal D)]\leq n.$$
Hence
$$
\deg(\eta)= [\bbQ(\eta):\bbQ]\leq [\bbQ(\mathcal D,\eta):\bbQ]= [\bbQ(\mathcal D,\eta):\bbQ(\mathcal D)][\bbQ(\mathcal D):\bbQ]\leq dn,
$$
where the second equality is a basic property about the degree of field extension (see e.g. ~ \cite[Proposition 1.20]{Morandi1996}).
 Since $ \lambda $ is transcendental and $ \eta $ is algebraic, it follows that  $ \abs{\lambda - \eta} > 0 $. Let $ n' $  be the unique integer  such that
\begin{equation}\label{eq:choose-n'}
	(2M)^{-d(n' + 1)} \leq \abs{\lambda-\eta}  < (2M)^{-dn'}.
\end{equation}
Then by this and the second inequality in \eqref{eq:eta-property},
\begin{equation*}
	d(n' + 1) \log(2M) \geq -\log\abs{\lambda-\eta} > n^{3}.
\end{equation*}
Thus for $ n_{0} $ large enough,
\begin{equation}\label{eq:degLemMulti}
	n' \geq (M+2) (d n)^{2}.
\end{equation}

Applying Theorem \ref{thm:SuperExpClose} with $ \delta = (2M)^{-3Md} $, we  see that  for $ n_{0} $ large enough there exists $ 0\neq P \in \calP^{n'} $ such that $ \abs{P(\lambda)} \leq (2M)^{-3Mdn'} $. Then by \eqref{eq:choose-n'},
\begin{equation*}
	\abs{\lambda-\eta} \geq (2M)^{-d(n'+1)} \geq (2M)^{dn'}\abs{P(\lambda)}^{1/M} \geq (2M)^{dn'/M}\abs{P(\lambda)}^{1/M}.
\end{equation*}
Combining this with \eqref{eq:choose-n'} yields
\begin{equation*}\label{eq:speed-n'}
	 (2M)^{dn'/M}\abs{P(\lambda)}^{1/M} \leq \abs{\lambda-\eta} \leq (2M)^{-dn'}.
\end{equation*}
  From this, \eqref{eq:degLemMulti} and  Proposition \ref{prop:algLargeOrder}  (in which we replace $n$ by  $d n $, and $k$ by  $ M $), we conclude that $ \eta $ is a zero of $ P $ of order at least $ M $.

On the other hand, since $ \abs{\eta} < (1+\abs{\lambda}) / 2 <1$,  it follows from Proposition \ref{lem:NumRoots} that there exists  $ N > 0 $ (only depending on $t_1,\ldots, t_m$ and $\lambda$) such that $ \eta $ is a zero of $ P $ of order at most $ N $. This leads to a contradiction by letting $ M > N $.
\end{proof}

\section{Proof of Theorem \ref{thm:EAA}} \label{sec:pf-EAA}

\begin{proof}[Proof of  Theorem \ref{thm:EAA}]
The proof  is almost identical to that of \cite[Theorem A.2]{RapaportVarju2024}, except some of the involved auxiliary results in Appendix A.1 of  \cite[Theorem A.2]{RapaportVarju2024} should be slightly revised.

Indeed, Theorem \ref{thm:EAA} is proved in \cite[Theorem A.2]{RapaportVarju2024} in the setting of homogeneous IFSs on $\bbR$ with positive contraction ratio and rational translations. Actually, the proof of \cite[Theorem A.2]{RapaportVarju2024} is based on some auxiliary results (see Propositions A.3, A.5 and A.7, Lemmas A.4 and A.6 in \cite{RapaportVarju2024}), which adapts and generalizes the main results of \cite{BreuillardVarju2019}. As was pointed out in \cite[Appendix A.1]{RapaportVarju2024}, the statements of \cite[Proposition A.3, Lemmas A.4 and A.6]{RapaportVarju2024} hold in the general setting of homogeneous IFSs on $\bbR$. Notice that  \cite[Proposition A.5 and A.7]{RapaportVarju2024} are built on a separation property (see \cite[Lemma 4.1]{RapaportVarju2024}) for roots of integer polynomials. By a similar separation property (see Proposition \ref{prop:SepRoots}) for roots of polynomials with algebraic coefficients, we are able to extend \cite[Proposition A.5 and A.7]{RapaportVarju2024} to the general setting of Theorem \ref{thm:EAA}; see Propositions \ref{prop-n1} and \ref{prop-n2}. Using these (revised) auxiliary results, the exact arguments of \cite[Theorem A.2]{RapaportVarju2024} give the proof Theorem \ref{thm:EAA}. Since the proof of \cite[Theorem A.2]{RapaportVarju2024} is rather long, we do not repeat it here and we refer to \cite{RapaportVarju2024} for the full details.
\end{proof}

Below we present two propositions, which are the extensions of \cite[Propositions A.5 and  A.7]{RapaportVarju2024} to our setting respectively.

Following \cite[Section 2]{BreuillardVarju2019} and \cite[Section 2]{Varju2019a}, we introduce the notion of  entropy at a given scale.
Let $ X $ be a bounded random variable in $\bbR$ with law $\nu$. For $ r > 0 $, the \textit{entropy of $ X $ at scale $ r $} is defined by
\begin{equation*}\label{eq:def-AvgEntropy}
	H(X; r) := \int_{0}^{1} H(\fl{ X/r + t }) \, dt,
\end{equation*}
 where $\fl{x}$ stands for the largest integer not exceeding $x$. We also write $H(\nu;r)$ in place of $H(X;r)$.
 As was noted in \cite{BreuillardVarju2019,Varju2019a}, the idea of the above averaging procedure originates in Wang's paper \cite{Wang2011}.

\begin{proposition}
\label{prop-n1}
For every $\varepsilon\in (0,1/2)$ there exists $C>0$ (depending on $\varepsilon$, $t_1,\ldots, t_m$ and $p_1,\ldots, p_m$) such that the following holds for all $n\geq N(\varepsilon,C)$. Let $0<r<n^{-Cn}$ and $\lambda \in (\varepsilon-1,-\varepsilon)\cup(\varepsilon, 1-\varepsilon)$ be given, and suppose that
$\frac{1}{n} H(\mu_\lambda^{(n)}; r)<H(p)$. Then there exists $0\neq \eta\in (-1,1)$, which is a root of a nonzero polynomial in ${\mathcal P}^n$ such that
$$
|\lambda-\eta|<r^{1/C}\quad \mbox{ and }\quad H(\mu_\eta^{(n)})\leq H(\mu_\lambda^{(n)}; r).
$$
\end{proposition}
\begin{proof}
Here we follow the arguments in the proof of \cite[Proposition A.5]{RapaportVarju2024} with slight modifications.

Let $\varepsilon\in (0,1)$ and let $C>1$ be large with respect to $\varepsilon$. Let  $n\geq $ be large with $C$ and let $r$ and $\lambda$ be given as in the statement of the proposition.

Let $0\leq s\leq 1$ be with
\begin{equation}
\label{e-temp1}
H\left( \left\lfloor r^{-1}\sum_{k=0}^{n-1} \xi_k\lambda^k+s \right\rfloor \right)\leq H(\mu_{\lambda}^{(n)};r)<nH(p),
\end{equation}
where $(\xi_k)_{k=0}^{n-1}$ is a finite sequence of independent random variables with common law $\nu=\sum_{j=1}^m p_j\delta_{t_j}$.
Let $\mathcal A$ be the set of all nonzero $P\in \mathcal P^n$ with $|P(\lambda)|\leq r$. By \eqref{e-temp1},
$\mathcal A$ is nonempty.

Given $P\in \mathcal A$ it follows that
Corollary \ref{cor-temp1}
that there exists $\eta_{P}\in\bbC$ with $P(\eta_{P})=0$ and
\[
|\eta_{P}-\lambda|\le(2^{n}\varepsilon^{-n}r)^{C^{-1/4}}.
\]
From $r<n^{-Cn}$, since $C$ is large with respect to $\varepsilon$, and
since $n$ is large with respect to $C$, it follows that we may assume
$|\lambda-\eta_{P}|<r^{C^{-1/2}}$.

For $Q,P\in\mathcal{A}$,
\[
|\eta_{P}-\eta_{Q}|\le|\eta_{P}-\lambda|+|\lambda-\eta_{Q}|\le2r^{C^{-1/2}}<2n^{-C^{1/2}n}.
\]
Thus, by Proposition \ref{prop:SepRoots} and by assuming
that $C$ is large enough, it follows that $\eta_{P}=\eta_{Q}$. Write
$\eta$ for this common value, then $P(\eta)=0$ for all $P\in\mathcal{A}$.
From this, \eqref{e-temp1} and by the definition
of $\mathcal{A}$,
\[
H(\mu_{\eta}^{(n)})\le H(\mu_{\lambda}^{(n)};r).
\]

Since $\lambda\in\bbR$ we have $|\overline{\eta}-\lambda|=|\eta-\lambda|$,
and so $|\eta-\overline{\eta}|\le2n^{-C^{1/2}n}$. For $P\in\mathcal{A}$
we clearly have $P(\overline{\eta})=0$. Thus, another application
of Proposition \ref{prop:SepRoots} gives $\eta=\overline{\eta}$.
Since $\lambda\in (\varepsilon-1 ,-\varepsilon)\cup (\varepsilon ,1-\varepsilon)$ we may assume $0\neq \eta\in(-1,1)$,
which completes the proof of the proposition.
\end{proof}

\begin{proposition}
\label{prop-n2}
For every $\varepsilon\in (0,1/2)$ there exists $C>0$ (depending on $\varepsilon$, $t_1,\ldots, t_m$ and $p_1,\ldots, p_m$) such that the following holds for all $n\geq N(\varepsilon,C)$. Let $\lambda \in (\varepsilon-1,-\varepsilon)\cup(\varepsilon, 1-\varepsilon)$  and suppose that
 there exists $0\neq \eta\in \bbC$, which is a root of a nonzero polynomial in ${\mathcal P}^n$, such that
$
|\lambda-\eta|<n^{-Cn}$. Then $\frac{1}{n} H(\mu_\lambda^{(n)}; r)=H(p)$ for all $r\leq |\lambda-\eta|^C$.
\end{proposition}
\begin{proof}
Here we repeat the proof of \cite[Proposition A.7]{RapaportVarju2024} with trivial modifications.  Let $\varepsilon\in (0,1/2)$, and let $C>1$ be large with respect to  $\varepsilon$,
let $n\ge1$ be large with respect to $C$, and let $\lambda$ and
$\eta$ be as in the statement of the proposition.

Suppose to the contrary
that there exists $0<r\le|\lambda-\eta|^{C}$ with $\frac{1}{n}H(\mu_{\lambda}^{(n)};r)<H(p)$.
By Corollary \ref{cor-temp1}, there exists
$\eta'\in(0,1)$, which is a root of a nonzero polynomial in $\mathcal{P}^{n}$,
such that $|\lambda-\eta'|<r^{1/C}\le|\lambda-\eta|$. In particular
$\eta\ne\eta'$ and
\[
|\eta-\eta'|\le|\eta-\lambda|+|\lambda-\eta'|\le2n^{-Cn}.
\]
However, as $C$ is assumed to be large enough,
this contradicts Proposition \ref{prop:SepRoots}, which
completes the proof of the proposition.
\end{proof}

\section{Proof of Theorem \ref{thm:UBGarsia}}\label{sec:UBGarsia}

In this section, we prove the following slight extension of Theorem \ref{thm:UBGarsia}, in which we allow $t_1,\ldots, t_m$ and $\eta$ to be complex algebraic numbers.

\begin{theorem} \label{thm:UBGarsiaC}  Let $t_1,\ldots, t_m$ be algebraic numbers in $\bbC$, and let $p=(p_j)_{j=1}^m$ be a probability vector. Set $\nu=\sum_{j=1}^m p_j\delta_{t_j}$. Then for  any $h\in (0,  H(p)) $, there is a positive number $ M  $ depending on $h$,  $ t_1,\ldots, t_m$ and  $p_1,\ldots, p_m$, such that $$ h_{\eta, \nu} \geq h $$ for every  algebraic number $\eta\in \bbC$ with $0<|\eta|<1$ and  $ M(\eta)\geq M $, where $h_{\eta, \nu}$ and $M(\cdot)$ are defined as in \eqref{e-Garsia} and \eqref{e-Mahler} respectively.
\end{theorem}

The proof of the above theorem is based on a lower bound estimate of the involved Garsia entropy.   For a  probability measure $\nu$ on $\bbC$ with bounded support,  following \cite{BreuillardVarju2020} we define
\begin{equation}\label{eq:def-Phi_nu}
	\Phi_{\nu}(a)=\sup_{t>0} \left\{  H(t a \xi + G)- H(t \xi + G) \right \} \mFor a \in \bbR,
\end{equation}
where $\xi$ is a complex random variable with law $\nu$, $ G $ is a complex Gaussian random variable independent of $ \xi $, with density $ \exp(-\abs{z}^{2}/2)/(2\pi)$ for $ z \in \bbC $, and $ H(\cdot) $ stands for  differential entropy (see \eqref{eq:def-DiffEntropy}).

The following proposition gives a lower bound of  $ h_{\eta, \nu}$,
 which extends a  result  of Breuillard and Varj\'u \cite[Proposition 13]{BreuillardVarju2020},
\begin{proposition}\label{prop:LB_Phi}
	Let $ \bbK $ be a subfield of $ \bbC $. Let $ \eta \in \bbC $ such that $ 0 < \abs{\eta} < 1 $ and $ [\bbK(\eta) : \bbK] < \infty $. Let $ \{u_{j}\}_{j=1}^m\subset \bbK$ and $ p = (p_{j})_{j=1}^m $ be a probability vector. Set $ \nu =\sum_{j=1}^mp_{j}\delta_{u_{j}}$. Then $ h_{\eta, \nu} \geq \Phi_{\nu}\left(\widetilde{M}_{\bbK}(\eta
	)\right) $, where
\begin{equation}\label{eq:def-wtM}
	\wt{M}_{\bbK}(\eta) = \prod_{ \eta_j} \frac{1}{ \min \{ 1 , \abs{ \eta_j } \} },
\end{equation}
with the product running over all the conjugates (including $\eta$ itself) of $\eta$ over $\bbK$.
\end{proposition}

Breuillard and Varj\'u \cite[Proposition 13]{BreuillardVarju2020} only proved the above result in the case when $ \bbK=\bbQ$. Before proving Proposition \ref{prop:LB_Phi}, we first apply it to the proof of Theorem \ref{thm:UBGarsiaC}.

\begin{proof}[Proof of Theorem \ref{thm:UBGarsiaC} assuming Proposition \ref{prop:LB_Phi}]
Let $t_1,\ldots, t_m, p_1,\ldots, p_m, \nu$ and $h$ be given as in the statement of the theorem. Set
$$\mathcal D=\{t_i-t_j:\; 1\leq i,j\leq m\}.$$ For $n\in \bbN$ let $\mathcal P^n$ denote the collection of polynomials in variable $X$ of degree not exceeding $n$ and with coefficients in $\mathcal D$.

By Lemma \ref{lem-extension}, we can choose an algebraic integer $\theta$ such that $\bbQ(\theta)=\bbQ({\mathcal D})$. Then we can take a suitable $L\in \bbN$ such that
\begin{equation*}
\label{e-integer_new}
Ls \in \bbZ[\theta] \quad \mbox{ for all }s\in \mathcal D.
\end{equation*}
Let $d=\deg(\theta)$, and  let $\theta_1,\ldots, \theta_d$ be the algebraic conjugates of $\theta$, with $\theta_1=\theta$. For $k=1,\ldots, d$, let $$\sigma_k: \bbQ(\theta)\to \bbQ(\theta_k)$$ be the field isomorphism such that $\sigma_k(a)=a$ for all $a\in \bbQ$ and $\sigma_k(\theta)=\theta_k$.

Let $\eta\in \bbC$ be an algebraic number with $0<|\eta|<1$. We aim to show that $h_{\eta,\nu}\geq h$ if the Mahler measure  of $\eta$ is large enough. For this purpose, we may assume that  that $P(\eta)=0$ for some  $P=\sum_{i=0}^na_i X^i\in \mathcal P^n$ with $n\in \bbN$; otherwise  $ h_{\eta,\nu} = H(p)>h $ and there is nothing left to prove.

 Fix such a polynomial $P=\sum_{i=0}^na_i X^i$ and define $F\in \bbC[X]$ by
\begin{equation}
\label{e-6.4}
F=L^d \prod_{k=1}^d \left(\sum_{i=0}^n \sigma_k(a_i) X^i\right)= \prod_{k=1}^d \left(\sum_{i=0}^n \sigma_k(La_i) X^i\right).
\end{equation}
Then $F\in \bbZ[X]$, which follows from
 the same argument as in the last paragraph of the proof of Lemma \ref{lem:ToUseMahler}.  Clearly $F(\eta)=0$.
 Writing
 $F=\sum_{i=0}^{dn} \beta_iX^i$ and letting $\tau$ be the smallest integer so that $\beta_\tau\neq 0$, by \eqref{e-6.4} we see that $\beta_\tau\leq C$, where
 $$ C:=L^d \max\{|\sigma_k(s)|^d:\; 1\leq k\leq d, \;s\in \mathcal D\}.$$

We claim that  \begin{equation}
\label{e-MM}
 \wt{M}_{\bbQ}(\eta)\geq C^{-1} M(\eta),
 \end{equation}
  where $\wt{M}_{\bbQ}(\eta) $ is defined as in \eqref{eq:def-wtM}. To see this, let $ f = \sum_{i=0}^{q} b_{i}X^{i} $ be the minimal polynomial of $ \eta $ in $ \bbZ[X] $. Then $ f $ divides $ F $ in $ \bbZ[X] $ since $ F(\eta) = 0 $.
	It follows that $b_0$ divides $\beta_\tau$. Consequently, $|b_0|\leq |\beta_\tau|\leq C$.
	By Vieta’s formulas,
	\begin{equation}\label{eq:M_MTilde}
		\frac{ M(\eta) }{ \wt{M}_{\bbQ}(\eta)  } = \abs{b_{q}}\prod_{i=1}^q \abs{\alpha_i} = \abs{b_{0}} \leq   C,
	\end{equation}
where $\alpha_1,\ldots, \alpha_q$ are the roots of $f$. This yields \eqref{e-MM}.

Now	let $ g =\sum_{i=0}^r c_iX^i$  be a minimal polynomial of $ \eta $ in $ \bbQ(\theta)[X] $ with coefficients in $ \bbZ[\theta] $.  For $ 1 \leq k \leq d $, define
\begin{equation}
\label{e-gk}
 g_{k} = \sigma_k(g):=\sum_{i=0}^r \sigma_k(c_i)X^i.
 \end{equation}
Since $ g $ is irreducible in $ \bbQ(\theta)[X] $, $ g_{k} $ is irreducible in $ \bbQ(\theta_{k})[X] $ for each $1\leq k\leq d$.
 Define
$$ G = \prod_{k=1}^{d} g_{k}.$$
 Then $ G \in \bbZ[X] $, which follows from
 a similar  argument as in the last paragraph of the proof of Lemma \ref{lem:ToUseMahler}. Since $G(\eta)=0$ and   $ f $ is the minimal polynomial of $ \eta $ in $ \bbZ[X] $, it follows that $f$ divides $G$.

 For a nonzero polynomial $ Q=\sum_{i=0}^l e_i X^i=e_l\prod_{i=1}^l(X-\gamma_l)$ in $\bbC[X]$, define
	 \begin{equation*}
		\wt{M}(Q) = \prod_{1\leq i\leq l:\; \gamma_i\neq 0 } \frac{1}{\min\{1, \abs{\gamma_i}\} },
	\end{equation*}
	with convention $ \wt{M}(Q) = 1 $ if $Q$ does not have nonzero root.
	It is easily checked that  $ \wt{M}(Q_{1}Q_{2}) = \wt{M}(Q_{1}) \wt{M}(Q_{2})$ for any $Q_1, Q_2\in \bbC[X]$,  and  $ \wt{M}( Q_{2})\geq \wt{M}(Q_{1}) $ if $ Q_{1} $ divides $ Q_{2}$.

For $1\leq k\leq d $, let $\eta_{k} $ be a root of $ g_{k} $. Using the above notation, we have
	\begin{equation*}\label{eq:M_F}
		\prod_{k=1}^{d} \wt{M}_{\bbQ(\theta_{k})}(\eta_{j}) = \prod_{k=1}^{d} \wt{M}(g_{k}) = \wt{M}(G) \geq \wt{M}(f) = \wt{M}_{\bbQ}(\eta),
	\end{equation*}
	where the first equality holds since $ g_{j} $ is irreducible in $ \bbQ(\theta_{j})[X] $, the second equality holds by $ G = \prod_{k=1}^{d} g_{k}  $, the inequality is by $ f $ dividing $ G $, and the last equality holds since $ f \in \bbZ[X] $ is the minimal polynomial of $\eta$  over $\bbQ$.  Thus there exists  some $ 1 \leq k_{0} \leq d $ such that
	\begin{equation*}\label{eq:j_{0}}
		\wt{M}_{\bbQ(\theta_{k_{0}})}(\eta_{k_{0}}) \geq (\wt{M}_{\bbQ}(\eta))^{1/d}.
	\end{equation*}
 From this and  \eqref{eq:M_MTilde}, we conclude that
	\begin{equation}\label{eq:wtM-eta-LB}
		\wt{M}_{\bbQ(\theta_{k_{0}})}(\eta_{k_{0}}) \geq \frac{M(\eta)^{1/d}}{C} > 1
	\end{equation}
	when $ M(\eta) $ is sufficiently large depending on $ \calD $.  Due to this and  \eqref{eq:def-wtM}, replacing $\eta_{k_{0}}$ by another root of $g_{k_0}$ if necessary, we may assume that $ \abs{\eta_{k_{0}}} < 1 $.

	Define $ \wt{t}_{j} = \sigma_{k_{0}}(t_{j}) $ for $1\leq j \leq m$. Let $ \wt{\nu} = \sum_{j=1}^m p_{j} \delta_{\wt{t}_{j}}$ and let $ (\wt{\xi}_{j})_{j=0}^{\infty} $ be  a sequence of i.i.d.\ random variables with common law $ \wt{\nu} $. Next we show that
	\begin{equation}
	\label{e-heta}
		h_{\eta,\nu} = h_{\eta_{k_{0}}, \wt{\nu}}.
	\end{equation}
To see this, it is sufficient to prove that for each polynomial $ Q
 $ with coefficients in $\mathcal D$,
 \begin{equation}
 \label{e-equiv}
 Q(\eta)=0  \Longleftrightarrow \sigma_{k_0}(Q)(\eta_{k_0})=0,
 \end{equation}
 where $\sigma_{k_0}(Q)$ is defined in a way similar to the definition of $\sigma_k(g)$ (see \eqref{e-gk}). Below we only prove the direction `$\Longrightarrow$' in \eqref{e-equiv}, the other direction follows from a similar argument.

 	Suppose that $ Q(\eta) = 0 $ for some polynomial $ Q
 $ with coefficients in $\mathcal D$.  Since  $\mathcal D\subset  \bbQ(\theta) $ and $ g $ is a minimal polynomial of $ \eta $ in $ \bbQ(\theta) $, $ g $ divides $ Q $ in $ \bbQ(\theta)[X] $. It follows that  $ g_{k_{0}}$ divides $ \sigma_{k_{0}}(Q) $ in $ \bbQ(\theta_{k_{0}})[X] $. Hence $ \sigma_{k_{0}}(Q)( \eta_{k_{0}} ) = 0 $ by $ g_{k_{0}}( \eta_{k_{0}} ) = 0 $.
 This proves the direction `$\Longrightarrow$' in \eqref{e-equiv}.

  By \eqref{e-heta} and  Proposition \ref{prop:LB_Phi} applied to $ h_{\eta_{k_{0}},\wt{\nu}} $, we obtain that
	\begin{equation}\label{eq:LB-hLambda}
		h_{\eta,\nu} = h_{\eta_{k_{0}}, \wt{\nu}} \geq \Phi_{\wt{\nu}}\left (\wt{M}_{\bbQ(\theta_{k_{0}})}(\eta_{k_{0}})\right ),
	\end{equation}
	where $\Phi_{\wt{\nu}}(\cdot)$ is as defined in \eqref{eq:def-Phi_nu} (in which we replace $\nu$ by $\wt{\nu}$).
	
 Let $ G $ be a complex Gaussian random variable with density $ \exp(-\abs{z}^{2}/2)/(2\pi)$ for $ z \in \bbC $. It is easily checked that
	\begin{equation*}
		\lim_{a\to \infty} H(a^{-1/2} \zeta + G) = H(G) \mAnd \lim_{a\to \infty} H(a^{1/2} \zeta + G)= H(G)+H(p),
	\end{equation*}
	where $ \zeta $ is a random variable with law $ \wt{\nu} $ and independent of $ G $.  Taking $ t = a^{-1/2} $ in \eqref{eq:def-Phi_nu} (in which we replace $\nu$ by $\wt{\nu}$)  yields that
	\begin{equation*}\label{eq:LimPhiA}
		\lim_{a\to\infty}  \Phi_{\wt{\nu}}(a) \geq \lim_{a\to \infty} H(a^{1/2} \zeta + G)- H(a^{-1/2} \zeta + G) = H(p).
	\end{equation*} % A depends $ \xi $ which in terms depends on $ \calD $
	Since $ h < H(p) $, there exists $ A  > 0 $ such that $  \Phi_{\wt{\nu}}(a) \geq h $ for all $ a \geq A $. By \eqref{eq:wtM-eta-LB},  $ \wt{M}_{\bbQ(\theta_{k_{0}})}(\eta_{k_{0}}) \geq A $  if $ M(\eta) >(A C)^d$. Then \eqref{eq:LB-hLambda} implies  that
	$$h_{\eta,\nu} \geq h$$
 if $ M(\eta) >(A C)^d$, 	which finishes the proof of the theorem.
\end{proof}

In the remaining part of this section, we prove Proposition \ref{prop:LB_Phi}. This is done by  adapting  the ideas of \cite[Proposition 13]{BreuillardVarju2020}.

 Let   $ \MatC $, $ \MatR[d] $,  $\GLR[d]$, $ \Ugrp{d} $ and $ \Ogrp{d} $  denote  the sets of $d\times d$  complex, real, real invertible, unitary and orthogonal matrices, respectively. Since  $\bbK$ in Proposition \ref{prop:LB_Phi} is assumed to be a subfield of $\bbC$, to apply the ideas of \cite[Proposition 13]{BreuillardVarju2020}, we need to transfer certain series of complex  matrices into the series of real matrices.  For this purpose,  let us introduce a natural operator  transferring complex matrices into real matrices.

Define $ \wh{\phantom{\cdot}\cdot\phantom{\cdot}} \colon \MatC[d] \to \MatR[2d] $ by
\begin{equation}\label{eq:matHat}
	\wh{A} = \Big(\Big[ \begin{smallmatrix}	a_{jk} & -b_{jk}\\ b_{jk} & a_{jk} 	\end{smallmatrix}\Big]\Big)_{1\leq j,k\leq d}
\end{equation}
if $ A = (a_{jk}+i b_{jk})_{1\leq j,k\leq d}$, where  $a_{jk}, b_{jk} \in \bbR$.
For $z\in \bbC$,  write $\wh{z}:=\wh{zI_d}$, where $I_d$ stands for the $d\times d$ identity matrix.

\begin{lemma}\label{lem:complex2real}
	The operator $ \wh{\phantom{\cdot}\cdot\phantom{\cdot}} $  has the following properties.
	\begin{enumerate}[(i)]
		\item\label{itm:C2R-AlgHomo} $ \wh{\phantom{\cdot}\cdot\phantom{\cdot}} $ is an injective $ \bbR$-algebra homomorphism, that is, for $ \alpha \in \bbR, A, B \in \MatC $,
		\begin{equation*}\label{eq:R-Algebra}
			\wh{\alpha A + B} = \alpha \wh{A} + \wh{B} \mAnd \wh{AB} = \wh{A}\wh{B}.
		\end{equation*}
		
		\item\label{itm:C2R-isometry} $ \wh{\phantom{\cdot}\cdot\phantom{\cdot}} $ is an isometry with respect to the norms induced by the standard inner products respectively, that is, $\norm{\wh{A}} = \norm{A} $ for $ A \in \MatC[d] $. If $ U \in \Ugrp{d} $, then $ \wh{U} \in \Ogrp{2d} $.
		
		\item\label{itm:C2R-commute} For any $ z \in \bbC $ and $ B \in \MatC[d] $,
		\begin{equation}\label{eq:wh-comm}
			\wh{z} \wh{B} = \wh{zB}  = \wh{B} \wh{z}.
		\end{equation}
	\end{enumerate}
\end{lemma}
\begin{proof} Consider the natural map $ \tau \colon \bbC^{d} \to \bbR^{2d} $ defined by
	\begin{equation*}
		\tau \left ( x_{1} + i y_{1}, \ldots, x_{d}+i y_{d} \right) = (x_{1}, y_{1}, \ldots, x_{d}, y_{d}),
	\end{equation*}
where  $x_{j}, y_{j} \in \bbR$ for $1\leq j \leq d$. 	
	Then the statements of the lemma are readily justified from the fact that  $ \wh{A}x = \tau A \tau^{-1} x $ for $ A \in \MatC $ and $x \in \bbR^{2d}$. We leave the details to the reader.
\end{proof}

To prove Proposition \ref{prop:LB_Phi}, we  need to use some entropy quantities  introduced in \cite{BreuillardVarju2020}.
Let $ X, Y $ be two independent bounded random variables in $ \euclid[d] $. Following Breuillard and Varj\'{u} \cite{BreuillardVarju2020}, we define
\begin{equation}\label{eq:def-+gaussian}
	H(X; B) = H(X + G_{B}) - H(G_{B}) \mFor B \in \GLR[d],
\end{equation}
where $ G_{B} $ is a Gaussian random variable in $\bbR^d$, independent of $X$, with mean  $ 0 $ and covariance matrix $ BB^{t} $. Here $ B^{t} $ stands for the transpose of $B$. For $ B_{1},B_{2} \in \GLR[d] $, write
\begin{equation}\label{eq:def-DeltaGaussEntropy}
	H(X; B_{1} | B_{2}) = H(X; B_{1}) - H(X; B_{2}).
\end{equation}
It follows from \eqref{eq:Affine-DiffEntropy} that for each $B\in \GLR[d]$,
\begin{equation}\label{eq:DiffEntropy-Scaling}
	H(X; B_{1}) = H(BX; BB_{1})  \mAnd H(X; B_{1} | B_{2}) =  H(BX; BB_{1} | BB_{2}).
\end{equation}

The quantities defined above have the following properties, which are crucial for our proof.

\begin{lemma}[{\cite{BreuillardVarju2020}}]
Let $B_1, B_2\in \GLR[d]$ be such that $\|B_1x\|\leq \|B_2x\|$ for all $x\in \bbR^d$. Assume that $X,Y$ are two bounded independent random variables taking values in $\bbR^d$. Then
\begin{equation}\label{eq:subadd-entropy}
	H(X + Y; B_{1} ) \leq H(X; B_{1}) + H(Y; B_{1}),
\end{equation}
\begin{equation}\label{eq:RotSym-Gauss}
	H(X; B_1O) = H(X; B_1) \mFor O \in \Ogrp{d},
\end{equation}
and
\begin{equation*}
	H(X+Y ; B_{1} | B_{2}) \geq H(X; B_{1}| B_{2}).
\end{equation*}
As a consequence, for any $ B \in \GLR[d] $ with $\norm{B} \leq 1$ and any nonnegative integer $k$,
\begin{equation}\label{eq:ConvIncreaseEntropy}
	H\left(X + Y; B^{k+1} | B^{k} \right) \geq H\left(X; B^{k+1} | B^{k} \right).
\end{equation}
If in addition $X$ is discrete, then
\begin{equation}\label{eq:LB-BySubadd}
	H(X) \geq H(X; B_1).
\end{equation}
\end{lemma}
\begin{proof}
\eqref{eq:RotSym-Gauss} follows from that $G_{B_1O}$ and $G_{B_1}$ have the same distribution for each $O \in \Ogrp{d}$.
\eqref{eq:LB-BySubadd} follows from the definition of $H(X; B_1)$ and \cite[Equation (2.10)]{BreuillardVarju2020}. The other statements of the lemma come from \cite[Lemma 9]{BreuillardVarju2020}.
\end{proof}

Now we are ready to prove Proposition \ref{prop:LB_Phi}.

\begin{proof}[Proof of Proposition \ref{prop:LB_Phi}]
	We adapt the proof of \cite[Proposition 13]{BreuillardVarju2020}.
	
	First we make some preparations. Let $ \eta_{1},\dotsc, \eta_{d} $ (including $\eta$ itself) be all the conjugates of $ \eta $ over $ \bbK $ with modulus strictly less than $ 1 $. By \eqref{eq:def-wtM}, $ \wt{M}_{\bbK}(\eta) = 1/\abs{\eta_{1}\cdots \eta_{d}}$. Set
	$$ \alpha_{1} = \dotsb =\alpha_{d-1} = 1,\quad \alpha_{d} = \abs{\eta_{1} \cdots \eta_{d}}  =  \frac{1}{ \wt{M}_{\bbK}(\eta)}.$$
	By \cite[Theorem 3]{Horn1954} (see also \cite[Theorem 3.6.6]{HornJohnson1994}),  there exist $ A\in \MatC $ with eigenvalues $ \eta_{1}, \dotsc, \eta_{d} $ and $ U, V\in \Ugrp{d} $ such that $ A = UDV $, where $ D  = \diag (\alpha_{1}, \ldots, \alpha_{d}) $. Then the spectrum radius $ \rho(A) $  of $A$  is strictly less than $ 1 $, and the operator norm of $A$ (with respect to the standard inner product on $ \bbC^{d} $)  satisfies $ \norm{A} = 1 $. Recall  the operator $ \wh{\phantom{\cdot}\cdot\phantom{\cdot}}$ defined in \eqref{eq:matHat}. Applying Lemma \ref{lem:complex2real} gives $ \wh{A} = \wh{U} \wh{D} \wh{V} $, where $ \wh{U} , \wh{V} \in \Ogrp{2d} $, and $\big\| \wh{A} \big\|  = 1  $. Moreover  by \eqref{eq:matHat},
$$
\wh{D} = \diag \left(\underbrace{1, \ldots,1}_{2d-2},1/ \wt{M}_{\bbK}(\eta), 1/ \wt{M}_{\bbK}(\eta)\right ).
$$

	 Let $ (\xi_{k})_{k=0}^{\infty} $ be a sequence of i.i.d.\ complex random variables with common law $ \nu $. Since $ \abs{\eta} < 1 $ and $ \rho(A) < 1$, the random variables
	\begin{equation*}
		X_{\eta} := \sum_{k=0}^{\infty}\xi_{k} \eta^{k} \mAnd X_{A} := \sum_{k=0}^{\infty} \xi_{k} A^{k}
	\end{equation*}
	are well-defined and bounded. For $ n \in \bbN $, write
	\begin{equation*}\label{eq:def-RandomSums}
		X_{\eta}^{(n)} := \sum_{k=0}^{n-1}\xi_{k}\eta^{k} \mAnd X_{A}^{(n)} := \sum_{k=0}^{n-1} \xi_{k} A^{k}.
	\end{equation*}

We claim that 	
\begin{equation}\label{eq:liftup-NoCostEntropy}
		H(X_{\eta}^{(n)}) \geq H(X_{A}^{(n)}) \quad\mbox{ for all } n\in \bbN,
	\end{equation}
where $H(\cdot)$ stands for Shannon entropy (see \eqref{eq:def-Shannon}).
To see this, let $\{v_1,\ldots, v_d\}$ be a basis of $\bbC^d$, with $v_j$ being an eigenvector of $A$ corresponding to the eigenvalue $\eta_j$ for each $1\leq j\leq d$. 	Suppose that $ P(\eta) = 0 $ for some polynomial $ P=\sum_{i=0}^n a_iX^i \in \mathcal P^n  $, with coefficients in
$$
\mathcal D=\{u_i-u_j:\; 1\leq i,j\leq m\}.$$	
Since $ \eta_{1}, \ldots, \eta_{d} $ are conjugates of $ \eta $ over $ \bbK $ and  $ \{u_{i}\}_{i=1}^m \subset \bbK $, it follows that 	 $ P(\eta_{j}) = 0 $ for $ 1 \leq j \leq d $. Moreover for each  $ x \in \bbC^{d} $,  writing $ x = \sum_{j=1}^{d} x_{j}v_{j} $ with $x_j\in \bbC$, we see that $$ \sum_{i=0}^n a_iA^i x = \sum_{j=1}^{d} \sum_{i=0}^n a_i x_jA^i v_j =\sum_{j=1}^{d} \sum_{i=0}^n a_i x_j\eta_j^i v_j =\sum_{j=1}^d x_jP(\eta_{j})v_{j} = 0. $$ Hence $\sum_{i=0}^n a_iA^i=0$. This implies \eqref{eq:liftup-NoCostEntropy}.

	Let $ \wh{\phantom{\cdot}\cdot\phantom{\cdot}}\colon \MatC \to \MatR[2d]  $ be as defined in \eqref{eq:matHat}.
	By \ref{itm:C2R-AlgHomo} and \ref{itm:C2R-isometry} of  Lemma \ref{lem:complex2real},
	\begin{equation}\label{eq:wh-XAn}
		\wh{X_{A}^{(n)}} = \sum_{k=0}^{n-1} \wh{\xi_{k}} \wh{A}^{k}.
	\end{equation}
	Notice that for each $n\in \bbN$ and $x\in \bbR^{2d}$, $\wh{X_{A }^{(n)}} x$ is a discrete random variable  in $\bbR^{2d}$.  Below we show that
	\begin{equation}\label{eq:keep1stTerm}
		H\left( \wh{X_{A }^{(n)}} x \right) \geq n H\left(\wh{\xi}_{0} x; \wh{A} | I_{2d}\right) \quad \mbox{ for all }n \in \bbN,\, x \in \euclid[2d],
	\end{equation}
where  $H(\cdot, \cdot|\cdot)$  is the quantity defined as in \eqref{eq:def-DeltaGaussEntropy}, and  $ I_{2d} $ stands for  the identity matrix in $ \MatR[2d] $.
	
To see \eqref{eq:keep1stTerm}, write $$ \wh{X_{A, k}}  = \sum_{j=0}^{\infty} \wh{\xi_{k+j}} \wh{A}^{j}.$$
 Since  $\xi_j$ ($j\in \bbN$) are i.i.d., the random variable $ \wh{X_{A, k}} $ is independent of $ \wh{X_{A}^{(k)}} $ and has the same law as $ \wh{X_{A}}$. By \eqref{eq:wh-comm} and \eqref{eq:wh-XAn}, for each $ k \in \bbN $,
	\begin{equation}\label{eq:wh-convo}
		\wh{X_{A}} =  \sum_{i=0}^{k-1} \wh{\xi_{i}} \wh{A}^{i}+ \sum_{j=k}^{\infty} \wh{\xi_{j}} \wh{A}^{j} = \wh{X_{A}^{(k)}} + \wh{A}^{k} \sum_{j=0}^{\infty} \wh{\xi_{k+j}} \wh{A}^{j} = \wh{X_{A}^{(k)}} + \wh{A}^{k} \wh{X_{A, k}}.
	\end{equation}
	Hence for each $x\in \bbR^{2d}$,
	\begin{equation*}
		\begin{aligned}
			n & H\left(\wh{\xi}_{0} x; \wh{A} | I_{2d}\right) \\ & \leq \sum_{k=0}^{n-1} H\left( \wh{X_{A}}x; \wh{A} | I_{2d}\right) & (\text{by \eqref{eq:ConvIncreaseEntropy} and \eqref{eq:wh-convo}}) \\
			& = \sum_{k=0}^{n-1} H\left( \wh{A}^{k}\wh{X_{A}}x; \wh{A}^{k+1} | \wh{A}^{k} \right) & (\text{by \eqref{eq:DiffEntropy-Scaling}}) \\
			& \leq \sum_{k=0}^{n-1} H\left( \wh{X_{A}}x; \wh{A}^{k+1} | \wh{A}^{k} \right) & (\text{by \eqref{eq:ConvIncreaseEntropy} and \eqref{eq:wh-convo}}) \\
			& = H\left( \wh{X_{A}}x; \wh{A}^{n} \right) - H\left( \wh{X_{A}}x; I_{2d} \right) & (\text{by \eqref{eq:def-DeltaGaussEntropy}}) \\
			& \leq H\left( \wh{X_{A}^{(n)}}x; \wh{A}^{n} \right) + H\left( \wh{A}^{n}\wh{X_{A}}x; \wh{A}^{n} \right) - H\left( \wh{X_{A}}x; I_{2d} \right)  & (\text{by \eqref{eq:subadd-entropy} and \eqref{eq:wh-convo}})\\
			& = H\left( \wh{X_{A}^{(n)}}x; \wh{A}^{n} \right) & (\text{by \eqref{eq:DiffEntropy-Scaling}}) \\
			& \leq H\left(\wh{X_{A}^{(n)}}x\right) & (\text{by \eqref{eq:LB-BySubadd}}).
		\end{aligned}
	\end{equation*}
	This proves \eqref{eq:keep1stTerm}.
	
	Now we are ready to estimate $ h_{\eta, \nu} $.  By \eqref{e-Garsia}, \eqref{eq:liftup-NoCostEntropy} and  the injectivity of the operator $ \wh{\phantom{\cdot}\cdot\phantom{\cdot}} $,
	\begin{equation}
\label{eq:Garsia-LB}
		h_{\eta, \nu}  = \lim_{n\to\infty} \dfrac{H\left(X_{\eta}^{(n)}\right)}{n}
		 \geq \lim_{n\to\infty} \dfrac{H\left(X_{A}^{(n)}\right)}{n}
		=  \lim_{n\to\infty} \dfrac{H\left(\wh{X_{A}^{(n)}}\right)}{n}.
\end{equation}
For $n\in \bbN$ and $x\in \bbR^{2d}$,
\begin{align*}                                                                        \\
		\dfrac{H\left(\wh{X_{A}^{(n)}}\right)}{n} & \geq  \frac{H \left(\wh{X_{A }^{(n)}}  x \right)}{n} \\
		& \geq H\left(\wh{\xi}_{0}  x; \wh{A} | I_{2d}\right) & (\text{by \eqref{eq:keep1stTerm}})                      \\
		& =  H\left(\wh{\xi}_{0} x;  \wh{U} \wh{D} \wh{V}  | I_{2d}\right) & (\text{by } \wh{A} = \wh{U} \wh{D} \wh{V}) \\
		& =  H\left(\wh{U}^{-1}\wh{\xi}_{0} x;   \wh{D} \wh{V}  | \wh{U}^{-1}\right) & (\text{by \eqref{eq:DiffEntropy-Scaling}})                                                                      \\
		& =  H\left(\wh{U}^{-1}\wh{\xi}_{0}  x; \wh{D} \wh{V}\right)-H\left( \wh{U}^{-1} \wh{\xi_{0}}  x; \wh{U}^{-1}\right) & (\text{by  \eqref{eq:def-DeltaGaussEntropy}})\\
		& = H\left(\wh{U}^{-1}\wh{\xi_{0}} x; \wh{D} \right) -H\left( \wh{U}^{-1}\wh{\xi_{0}}  x; I_{2d}\right) & (\text{by \eqref{eq:RotSym-Gauss}}) \\
		& =  H\left(\wh{\xi}_{0}  \wh{U}^{-1} x; \wh{D} | I_{2d} \right) &  (\text{by \eqref{eq:def-DeltaGaussEntropy} and \eqref{eq:wh-comm}}) \\
		& = H\left(\wh{\xi}_{0}  y; \wh{D} | I_{2d} \right) & (\text{by taking } y = \wh{U}^{-1}x).
	\end{align*}
	From this and \eqref{eq:Garsia-LB}, we obtain that
	\begin{equation*}
		h_{\eta, \nu} \geq H\left(\wh{\xi}_{0}  x; \wh{D} | I_{2d} \right) \quad \mbox{ for all } x \in \euclid[2d].
	\end{equation*}
	For each $t>0$,  taking $ x = t e_{2d-1} =(0,\ldots, 0, t,0)$  in the above inequality gives	\begin{align*}
		h_{\eta, \nu} & \geq H\left( t \wh{\xi}_{0}  e_{2d-1}; \wh{D}\right) - H\left( t \wh{\xi}_{0} e_{2d-1}; I_{2d} \right) & (\text{by } \eqref{eq:def-DeltaGaussEntropy}) \\
		& = H\left( t \wh{D}^{-1} \wh{\xi}_{0}  e_{2d-1}; I_{2d}\right) - H\left( t \wh{\xi}_{0} e_{2d-1}; I_{2d} \right) & (\text{by } \eqref{eq:DiffEntropy-Scaling}) \\
		& = H\left(t\wh{D}^{-1}\wh{\xi}_{0}  e_{2d-1} + G_{I_{2d}}\right) -  H\left(t\wh{\xi}_{0}e_{2d-1} +  G_{I_{2d}}\right) & (\text{by \eqref{eq:def-+gaussian}}) \\
		& = H\left(t \wt{M}_{\bbK}(\eta) \zeta + G_{I_{2}}\right) -  H\left(t\zeta + G_{I_{2}}\right),
	\end{align*}
	where in the last equality $ \zeta $ is a random variable in ${\mathbb R}^2$ with law $ \sum_{j\in\Lambda}p_{j}\delta_{(a_{j}, b_{j})} $ in which  $a_j, b_j$ are the real and imaginary parts of $u_j$ (i.e. $ u_{j} = a_{j} + ib_{j}$), and $ \zeta $ is independent of $ G_{I_{2}} $. The last equality follows by integrating out the first $ 2d-2 $ variables using $ F(xy) = xF(y) + yF(x) $ where $ F(x) = -x \log x $. This finishes the proof by taking supremum over $ t > 0 $.
\end{proof}

\bigskip

\noindent {\bf Acknowledgements}. This work was partially supported by the General Research Funds (CUHK14305722, CUHK14308423)  from the Hong Kong Research Grant Council, and by a direct grant for research from the Chinese University
of Hong Kong.

% ---------------------------------------------------------------------------- %
%                                   reference                                  %
% ---------------------------------------------------------------------------- %
% Get my*.bst from https://github.com/zfengg/toolkit/tree/master/tex/bst
%\bibliographystyle{myplain}
%\bibliography{HomoAlgTrans_050324}\label{sec:ref}

\end{document}